\documentclass[a4paper,11pt]{amsart}

\usepackage{amsmath,amssymb,amsthm,mathtools,mathrsfs}
\usepackage[margin=25mm]{geometry}
\usepackage{enumitem}
\usepackage[colorlinks=true,linkcolor=blue,urlcolor=blue]{hyperref}
\usepackage{booktabs}
\usepackage{longtable}
\usepackage{array}
\usepackage{graphicx}
\newtheorem{theorem}{Theorem}[section]
\newtheorem{proposition}[theorem]{Proposition}
\newtheorem{lemma}[theorem]{Lemma}
\newtheorem{corollary}[theorem]{Corollary}
\theoremstyle{definition}
\newtheorem{definition}[theorem]{Definition}
\newtheorem{remark}[theorem]{Remark}

\newcommand{\Q}{\mathbb{Q}}
\newcommand{\Z}{\mathbb{Z}}
\newcommand{\Tr}{\operatorname{Tr}}
\newcommand{\Nm}{\operatorname{N}}
\newcommand{\Aut}{\operatorname{Aut}}
\newcommand{\PSL}{\operatorname{PSL}}

\newcommand{\GL}{\operatorname{GL}}

\newcommand{\Span}{\operatorname{Span}}
\newcommand{\Cl}{\operatorname{Cl}}

\title{Two Extremal Even Unimodular Lattices of Rank $88$}
\author{Ryotaro Sakamoto}
\date{25 September 2026}

\begin{document}

\begin{abstract}
In this paper, we construct two extremal positive definite  even unimodular lattices of rank $88$ using a Hermitian tensor product over $\mathbb{Q}(\sqrt{-23})$,  thereby extending the known existence range for such lattices from rank \(80\) to rank \(88\).
\end{abstract}

\maketitle

\section{Introduction}
\label{sec:introduction}

A positive definite even unimodular $\mathbb Z$-lattice $L$ can exist only if its rank $r$ is divisible by $8$.
Moreover, the bound
\[
 \min(L)\le 2\left\lfloor\frac{r}{24}\right\rfloor+2
\]
is known (cf.~\cite{MallowsOdlyzkoSloane1975}).
A positive definite even unimodular lattice for which equality holds in this bound is called \emph{extremal}.
Extremal lattices are highly constrained by the theory of modular forms, and their minimal vectors often exhibit remarkable arithmetic, geometric, and combinatorial structures, including spherical design properties.
On the other hand, extremal lattices do not exist in arbitrarily high rank: Mallows--Odlyzko--Sloane~\cite{MallowsOdlyzkoSloane1975} proved nonexistence in sufficiently high rank, and Jenkins--Rouse~\cite{JenkinsRouse2011} obtained the explicit bound $r\le 163264$. Thus, determining the ranks in which extremal lattices exist and constructing explicit examples are fundamental problems in lattice theory.

We briefly recall what is known about extremal lattices. In ranks at most $24$, the extremal lattices are precisely $E_8$, $E_8\perp E_8$, $D_{16}^+$, and the Leech lattice $\Lambda_{24}$. The situation changes  in higher ranks: there are more than $10^7$ pairwise nonisomorphic extremal lattices of rank $32$, and more than $10^{51}$ of rank $40$; see, for example, \cite{King2003,NebeVenkov2013}.

Leech--Sloane used ternary self-dual codes to construct two extremal lattices of rank $48$,
$P_{48p}$ and $P_{48q}$~\cite{LeechSloane1971}.
Subsequently, Quebbemann introduced a general lattice construction and used it to construct a rank $64$
extremal lattice~\cite{Quebbemann1984}.
Ozeki developed a method for constructing even unimodular lattices from ternary self-dual codes,
and in particular obtained an extremal lattice of rank $56$~\cite{Ozeki1989}.
Batut--Quebbemann--Scharlau also obtained another rank $56$ example by a construction using cyclotomic fields
\cite{BatutQuebbemannScharlau1995}.
Later, Nebe constructed a third rank $48$ extremal lattice by a cyclo-quaternionic construction,
namely $P_{48n}$~\cite{Nebe1998}.
Meanwhile, Bachoc--Nebe used a structure related to the Mathieu group $M_{22}$
to construct two rank $80$ extremal lattices of minimum $8$
\cite{BachocNebe1998}.
At that time, the existence of an extremal lattice of rank $72$ was still open, so rank $80$ examples had been found before any rank $72$ example.

Nebe subsequently filled this gap.
She worked over the ring of integers of $\mathbb{Q}(\sqrt{-7})$ and
considered the Hermitian tensor product of the Barnes lattice and the Leech lattice, thereby
constructing a rank $72$ even unimodular lattice $\Gamma_{72}$ of minimum $8$
\cite{Nebe2012}.
This established the existence of extremal  lattices
in every positive dimension divisible by $8$ up to $80$.
Subsequent examples include a fourth rank $80$ example due to Watkins~\cite{Watkins2012},
a fourth rank $48$ example $P_{48m}$ due to Nebe~\cite{Nebe2014},
and a new rank $64$ construction using a generalized quadratic residue code due to Shimada
\cite{Shimada2018}.
More recently, Shimada showed that Quebbemann's construction produces many  rank $64$ extremal lattices \cite{Shimada2022}.

Following Nebe's construction of the rank \(72\) extremal lattice \(\Gamma_{72}\), no extremal lattices of rank greater than \(80\) were known. In this paper, we construct two non-isomorphic extremal lattices of rank \(88\), thereby raising the largest rank for which an extremal lattice is known from \(80\) to \(88\).

\begin{theorem}[Corollary~\ref{cor:extremal} and Proposition~\ref{prop:L0-extremal}]
\label{thm:extremal88}
There exist two non-isomorphic positive definite even unimodular extremal lattices of rank $88$.
\end{theorem}

The constructions in this paper are inspired by Nebe's construction of $\Gamma_{72}$ using a Hermitian tensor product; see Remark~\ref{rem:description} for details. We construct two rank $88$ lattices by taking Hermitian tensor products of the rank $4$ Hermitian lattice $E$ with two rank $11$ Hermitian lattices $J$ and $J_0$ over $\mathbb Q(\sqrt{-23})$. The trace lattice of $E$ is isomorphic to $E_8$, while $J$ and $J_0$ arise from ideals of $\mathbb Q(\zeta_{23})$. The resulting lattices $(L,B)$ and $(L_0,B_0)$ are both extremal and are not isomorphic as $\mathbb Z$-lattices.

An advantage of the Hermitian tensor product construction is that the computations required to estimate the minimum can be reduced to problems involving lattices of comparatively small rank (cf.~\cite{Coulangeon2000,CoulangeonNebe2013}). The key ingredient is Coulangeon's method for bounding tensor norms in tensor products of Hermitian lattices using minimal discriminants of sublattices~\cite{Coulangeon2000}.

Since $88=3\cdot24+16$, Venkov's theorem~\cite{Venkov1984} implies that the minimal vectors of any extremal lattice of rank $88$ form a spherical $3$-design. In particular, both $(L,B)$ and $(L_0,B_0)$ are strongly eutactic. Although perfection does not follow from Venkov's theorem, direct computations show that both $(L,B)$ and $(L_0,B_0)$ are perfect. Hence both lattices are eutactic and perfect, and therefore extreme. Consequently, the densities of the corresponding lattice sphere packings are locally maximal; see Proposition~\ref{prop:extreme88}. On the other hand, Proposition~\ref{prop:not-4-design} shows that the minimal vectors of any extremal even unimodular lattice of rank $88$ do not form a spherical $4$-design.

\subsection*{Acknowledgment}
The author would like to thank Miyu Suzuki and Hiroyoshi Tamori for kindly explaining to him various results and problems in lattice theory.
The author was supported by JSPS KAKENHI Grant
Number JP24K16886.

\subsection*{Computational reproducibility}
 The exact inputs, source code, verification programs, and certificates for the computer-assisted calculations are available at \begin{center}
     \url{https://github.com/RyotaroSakamoto0113/rank-88-extremal-lattice-computations}. 
 \end{center}

\subsection*{Declaration on the Use of Artificial Intelligence}
ChatGPT was used in all computer-assisted calculations reported in this work. It was also used to create the code, computational data, certificates, and documentation released in the accompanying GitHub repository, and to assist in structuring this manuscript.

\subsection*{Notation}
We fix an embedding $\overline{\mathbb Q} \subset \mathbb{C}$ and, for each $n\ge1$, set $\zeta_n:=\exp(2\pi \sqrt{-1}/n)$.
For a number field $K \subset\overline{\mathbb Q}$, we write $\mathcal O_K$ for its ring of integers.
For a finite extension $L/K$, we write its trace and norm as
\[
\operatorname{Tr}_{L/K}\colon L\to K,
\qquad
\operatorname{N}_{L/K} \colon L^\times\to K^\times,
\]
respectively.
We set $\operatorname{N}_{L/K}(0) := 0$.
We also write the inverse different as $\mathfrak D_{L/K}^{-1}
:= \{x\in L \mid \operatorname{Tr}_{L/K}(x\mathcal O_L)\subset\mathcal O_K\}$.
If $K$ is a CM field, we write $K^+$ for its maximal totally real subfield.
We write $C_n$ for the cyclic group of order $n$.

\section{Hermitian lattices}\label{sec:generalities}

In this section, let $K$ be an imaginary quadratic field, and write $x\mapsto\bar x$ for its nontrivial automorphism.

\begin{definition}
A Hermitian form on an $\mathcal O_K$-module $L$ is a map
\[
h\colon L \times L \to  K
\]
that is $\mathcal O_K$-linear in the first variable and satisfies $h(y,x)=\overline{h(x,y)}$. It is called positive definite if $h(x,x)>0$ for every $x\ne0$.
A pair consisting of a finitely generated projective $\mathcal O_K$-module $L$ and a positive definite Hermitian form $h$ is called a Hermitian lattice.
Via the natural embedding, we always regard $L$ as a submodule of $L \otimes_{\mathcal{O}_K} K$, and define
\[
L^\# :=\{x\in L \otimes_{\mathcal{O}_K} K \mid h(x,L)\subset\mathcal O_K\}
\]
to be the Hermitian dual. We say that $(L,h)$ is self-dual if $h(L,L)\subset\mathcal O_K$ and $L^\#=L$.
\end{definition}

A Hermitian lattice $(L,h)$ of rank $r$ over $\mathcal O_K$ can be written, for ideals $\mathfrak a_1,\ldots,\mathfrak a_r$ of $\mathcal O_K$, as
\[
L=\mathfrak a_1e_1\oplus\cdots\oplus\mathfrak a_re_r.
\]
In this situation, the pairs $(\mathfrak a_i,e_i)$ form a pseudo-basis of $L$, and the ideal class $[\mathfrak a_1\cdots\mathfrak a_r]$ is called the Steinitz class of $L$.
Moreover, we define the discriminant $d_L$ of  $L$ by
\[
d_L :=
\left(\prod_{i=1}^r \operatorname{N}_{K/\mathbb Q}(\mathfrak a_i)\right)
\det\bigl(h(e_i,e_j)\bigr)_{i,j}.
\]
For an $\mathcal O_K$-submodule $P\subset L$, write its saturation as
$P^{\mathrm{sat}} := (K \otimes_{\mathcal{O}_K} P)\cap L$.

\begin{definition}
For a Hermitian lattice $(L,h)$ of rank $r$, define its projective minimum by
\[
\mu^p(L) := \min\{d_P \mid 0\ne P\subset L,\ \operatorname{rank}_{\mathcal O_K}P=1\}.
\]
For a fixed Steinitz type $c$, define the projective Hermitian constant by
\[
\gamma_{r,K}^{(c)} := \sup_{\text{type}(L)=c}\frac{\mu^p(L)}{d_L^{1/r}},
\]
where $L$ ranges over all Hermitian lattices of rank $r$ whose Steinitz class  is $c$.
For simplicity, we also put $\gamma_{r,K} := \gamma_{r,K}^{([\mathcal{O}_K])}$.
\end{definition}

Braun--Coulangeon give a method for computing the projective Hermitian  constants of $K$ in~\cite{BraunCoulangeon2015};
for $K = \mathbb{Q}(\sqrt{-23})$, the following values are known.

\begin{proposition}\label{prop:hermitian_constant}
If $K = \mathbb{Q}(\sqrt{-23})$,
then $\gamma_{2, \mathbb{Q}(\sqrt{-23})}^{(c)} = \sqrt{23/5}$ for every Steinitz type $c$.
Moreover, $\gamma_{3,  \mathbb{Q}(\sqrt{-23})} = \sqrt[3]{30613/1029}$,
and if $c$ is nontrivial, then $\gamma_{3,   \mathbb{Q}(\sqrt{-23})}^{(c)} = \sqrt[3]{46}$.
\end{proposition}
\begin{proof}
For $r=2$, this is due to Braun--Coulangeon~\cite[\S6.1, Table~1]{BraunCoulangeon2015}, and is also recorded by Braun~\cite[Appendix~A.6]{Braun2012}.
For $r=3$, the values were computed by Braun in~\cite[Appendix~B.2]{Braun2012}.
\end{proof}

\subsection{A useful inequality for tensor rank}

Let $(L,h)$ be a Hermitian lattice over $\mathcal O_K$.

\begin{definition}[{cf.~\cite[p. 120]{Coulangeon2000} and~\cite[p. 50]{CoulangeonNebe2013}}]
For $1\le s\le\operatorname{rank}_{\mathcal O_K}L$, we  define
\[
d_s^{\mathrm{proj}}(L)
:=\min\{d_M\mid M\subset L,\ \operatorname{rank}_{\mathcal O_K}M=s\}.
\]
Note that, unlike in the settings of \cite[p.~120]{Coulangeon2000} and \cite[p.~50]{CoulangeonNebe2013}, $M$ is allowed to be a nonfree $\mathcal{O}_K$-sublattice.
\end{definition}

\begin{definition}
For any $K$-vector spaces $V$, $W$ and $0\ne z\in V\otimes_KW$,
the tensor rank of $z$ is the least integer $s$ such that $z=\sum_{i=1}^s x_i\otimes y_i$.
\end{definition}

\begin{proposition}[{cf. \cite[Proposition 3.2]{Coulangeon2000}}]
\label{prop:inequality-tensor}
Let $(L,h) = (L_1,h_1)\otimes_{\mathcal O_K}(L_2,h_2)$.
If $0\ne z\in L$ has tensor rank $s$, then
\[
h(z,z)\ge s\left(d_s^{\mathrm{proj}}(L_1)d_s^{\mathrm{proj}}(L_2)\right)^{1/s}.
\]
\end{proposition}
\begin{proof}
For \(i\in\{1,2\}\), let \(V_i:=L_i\otimes_{\mathcal O_K}K\).
Define \(K\)-subspaces \(U_i\subset V_i\), for \(i=1,2\), by
\begin{align*}
  U_1 &:=\operatorname{Im}\left(
\operatorname{Hom}_K(V_2,K)\to V_1;
\ \varphi\mapsto(\operatorname{id}_{V_1}\otimes\varphi)(z)
\right),
\\
U_2 &:=\operatorname{Im}\left(
\operatorname{Hom}_K(V_1,K)\to V_2;
\ \varphi\mapsto(\varphi\otimes\operatorname{id}_{V_2})(z)
\right).
\end{align*}
Put \(M_i:=L_i\cap U_i\) for \(i=1,2\).
Since \(z\) has tensor rank \(s\), we have \(\dim_K U_i=s\), and $z\in U_1\otimes_K U_2$.
Since \(L_i/M_i\) is torsion-free, hence projective over \(\mathcal O_K\), the inclusions \(M_i\hookrightarrow L_i\) split. In particular,
\[
z\in M_1\otimes_{\mathcal O_K}M_2.
\]
Write
\begin{align*}
M_1&=\bigoplus_{i=1}^s\mathfrak a_i e_i^{(1)},
&
M_2&=\bigoplus_{j=1}^s\mathfrak b_j e_j^{(2)},
&
z&=\sum_{i,j=1}^s c_{ij}e_i^{(1)}\otimes e_j^{(2)},
\qquad c_{ij}\in\mathfrak a_i\mathfrak b_j,\\
C&:=(c_{ij}),
&
\mathfrak a&:=\prod_{i=1}^s\mathfrak a_i,
&
\mathfrak b&:=\prod_{j=1}^s\mathfrak b_j.
\end{align*}
Since \(z\) has tensor rank \(s\), we have $0\ne\det C\in\mathfrak a\mathfrak b$, and therefore
\[
|\det C|^2
\ge
\operatorname{N}_{K/\mathbb Q}(\mathfrak a)\operatorname{N}_{K/\mathbb Q}(\mathfrak b).
\]
For \(i=1,2\), let $H_i$  be the Gram matrix of  $U_i$ with respect to the \(K\)-basis \(e_1^{(i)},\ldots,e_s^{(i)}\).
Since \(H_i\) is a positive-definite Hermitian matrix, there exists \(P_i\in\operatorname{GL}_s(\mathbb C)\) such that $H_i=P_i^{\mathsf T} \, \overline{P_i}$.
Put $Z:=P_1CP_2^{\mathsf T}$.
Then \(h(z,z)=\operatorname{tr}(Z \overline{Z}^{\mathsf T})\). Since \(Z\overline{Z}^{\mathsf T}\) is positive definite, the arithmetic--geometric mean inequality applied to its eigenvalues gives
\[
h(z,z) \ge s\det(Z\overline{Z}^{\mathsf T})^{1/s} =s\left(|\det C|^2\det H_1\det H_2\right)^{1/s}.
\]
Moreover,
\[
|\det C|^2\det H_1\det H_2
\ge
\operatorname{N}_{K/\mathbb Q}(\mathfrak a)
\operatorname{N}_{K/\mathbb Q}(\mathfrak b)
\det H_1\det H_2
=
d_{M_1}d_{M_2} \geq d_s^{\mathrm{proj}}(L_1) d_s^{\mathrm{proj}}(L_2).
\]
This completes the proof.
\end{proof}

\section{\texorpdfstring{Definition of the rank $44$ Hermitian lattice}{Definition of the rank 44 Hermitian lattice}}
\label{sec:tensor-decomposition}

Throughout, let $F := \mathbb{Q}(\sqrt{-23})$, and put $\displaystyle \alpha := \frac{1+\sqrt{-23}}{2}$. Then $\mathcal O_F=\mathbb Z[\alpha]$.

\subsection{\texorpdfstring{The rank $4$ Hermitian lattice $E$}{The rank 4 Hermitian lattice E}}

Fix a totally positive generator $\delta_4 \in \mathcal{O}_{F(\zeta_5)^+}$ of $\mathfrak D_{F(\zeta_5)/\Q}$, i.e.,
$\delta_4\mathcal{O}_{F(\zeta_5)} = \mathfrak D_{F(\zeta_5)/\Q}$.
For example, one may take
\begin{align}\label{eq:delta4}
\delta_4 =
\frac{
11\sqrt{-23}(\zeta_5-\zeta_5^{-1})(\zeta_5+\zeta_5^{-1})
-2\sqrt{-23}(\zeta_5-\zeta_5^{-1})
-115(\zeta_5+\zeta_5^{-1})+115}{2}.
\end{align}

\begin{definition}\label{def:rank4lattice}
Let $E:=\mathcal O_{F(\zeta_5)}$.
Define a Hermitian form $h_4\colon E\times E\to F$ by
\[
h_4(u,v) := \operatorname{Tr}_{F(\zeta_5)/F}(\delta_4^{-1}u\bar v).
\]
\end{definition}

\begin{remark}
Lattices of this kind, obtained by equipping an ideal of a number field with a trace form, have been widely studied as ideal lattices. For related earlier work, see~\cite{BayerFluckiger1984,BachocBatut1992,BayerFluckigerMartinet1994,BatutQuebbemannScharlau1995,BayerFluckiger1999,BayerFluckiger2000,BayerFluckiger2002Determinants,BayerFluckigerSuarez2005,BayerFluckigerSuarez2006, Nebe2016}.
\end{remark}

\begin{remark}\label{rem:E-Gram}
We take $\delta_4$ as in~\eqref{eq:delta4}.
The $\mathcal O_F$-module $E=\mathcal O_{F(\zeta_5)}$ is free, with an $\mathcal O_F$-basis given by
\[
\begin{aligned}
p_1&:=(1-\alpha)+\alpha\zeta_5-2\zeta_5^3,
&\qquad
p_2&:=-2-2\zeta_5-3\zeta_5^3,\\
p_3&:=1-3\zeta_5-(1+\alpha)\zeta_5^3,
&\qquad
p_4&:=-\alpha-\zeta_5^2-(1+\alpha)\zeta_5^3.
\end{aligned}
\]
With respect to this basis, the Gram matrix of $h_4$ is given by
\[
\begin{pmatrix}
1&0&\dfrac{4}{\sqrt{-23}}&
\dfrac{\alpha}{\sqrt{-23}}\\
0&1&\dfrac{\alpha-1}{\sqrt{-23}}&
\dfrac{4}{\sqrt{-23}}\\
-\dfrac{4}{\sqrt{-23}}&
\dfrac{\alpha}{\sqrt{-23}}&1&0\\
\dfrac{\alpha-1}{\sqrt{-23}}&
-\dfrac{4}{\sqrt{-23}}&0&1
\end{pmatrix}.
\]
\end{remark}

\begin{proposition}\label{prop:rank4-algebraic-properties}
The following statements hold.
\begin{enumerate}[label=(\roman*)]
\item $(E,h_4)$ is a Hermitian lattice of rank $4$ over $\mathcal O_F$.
\item
$\displaystyle h_4(E,E)\subset\mathfrak D_{F/\Q}^{-1}
=\frac1{\sqrt{-23}}\mathcal O_F$. In particular, $h_4(z, z) \in \Q_{>0} \cap\mathfrak D_{F/\Q}^{-1}= \Z_{\ge 1}$ for any $0 \neq z \in E$.
\item $E^\#=\sqrt{-23}\,E$.
\item The trace form $B_4(x,y):=\operatorname{Tr}_{F/\Q}h_4(x,y)$ is a positive definite even unimodular integral form of rank $8$. Hence $(E,B_4)\cong E_8$.
\end{enumerate}
\end{proposition}

\begin{proof}
The Hermitian property follows immediately from the definition.  For $0\ne z\in E$, we have
\[
h_4(z,z)
=\sum_{\substack{\sigma \colon F(\zeta_5)\hookrightarrow\mathbb C\\ \sigma|_{F}=\mathrm{id}}}
\sigma(\delta_4)^{-1}|\sigma(z)|^2,
\]
and since $\delta_4$ is totally positive, $h_4$ is positive definite. This proves (i).
Next, since
$\delta_4^{-1}\mathcal O_{F(\zeta_5)}=\mathfrak D_{F(\zeta_5)/\Q}^{-1}$, for any $u,v\in E$ and $b\in\mathcal O_F$ we have
\[
\operatorname{Tr}_{F/\Q}\!\left(
 b\,\operatorname{Tr}_{F(\zeta_5)/F}(\delta_4^{-1}u\bar v)
\right)
=
\operatorname{Tr}_{F(\zeta_5)/\Q}(b\delta_4^{-1}u\bar v)\in\Z.
\]
Therefore $\operatorname{Tr}_{F(\zeta_5)/F}(\delta_4^{-1}u\bar v)
\in \mathfrak D_{F/\Q}^{-1}$, proving (ii).
By duality for the relative trace pairing, we also have $E^\#$
$=\delta_4\,\mathfrak D_{F(\zeta_5)/F}^{-1}=\sqrt{-23}\,E$, which proves (iii).
Finally, we prove (iv). Since $h_4(z,z)\in \Z_{\geq 0}$, we have  $B_4(z,z)=2h_4(z,z)\in2\mathbb Z$.
Moreover, since $\{z\mid B_4(z,E)\subset\Z\}=\{z\mid h_4(z,E)\subset\mathfrak D_{F/\Q}^{-1}\}=\{z\mid\sqrt{-23}\,z\in E^\#\}=E$, the lattice $(E,B_4)$ is unimodular.
It is also positive definite by (i), and the uniqueness of the positive definite even unimodular lattice of rank $8$ therefore gives $(E,B_4)\cong E_8$.
\end{proof}

\begin{remark}
The group $\operatorname{N}_{F(\zeta_5)/F(\zeta_5)^+}(\mathcal O_{F(\zeta_5)}^\times)$ coincides with the group of totally positive units of $\mathcal O_{F(\zeta_5)^+}$. Hence the isometry class of the Hermitian lattice $(E,h_4)$ is independent of the choice of the totally positive generator $\delta_4 \in \mathcal O_{F(\zeta_5)^+}$ of $\mathfrak D_{F(\zeta_5)/\Q}$.
\end{remark}

\begin{proposition}\label{prop:d_s(E)}
For the Hermitian lattice $(E,h_4)$, we have
\[
\left(d_1^{\mathrm{proj}}(E), d_2^{\mathrm{proj}}(E), d_3^{\mathrm{proj}}(E), d_4^{\mathrm{proj}}(E)\right) = \left(1,\frac5{23},\frac1{23},\frac1{529}\right).
\]
\end{proposition}
\begin{proof}
First, Proposition~\ref{prop:rank4-algebraic-properties}(iv) shows that the trace lattice $(E,B_4)$ is isomorphic to $E_8$.
Since the minimum of $B_4(z,z)$ is $2$ and $B_4(z,z)=2h_4(z,z)$, we obtain $d_1^{\mathrm{proj}}(E) \leq 1$.
For every $\mathcal O_F$-sublattice $\mathfrak a e\subset E$ of rank $1$, the determinant of the trace form on $\mathfrak a e$ is $23d_{\mathfrak a e}^2$. Since the trace form is integral and positive definite, we have $23d_{\mathfrak a e}^2\in\mathbb Z_{>0}$. Since $d_{\mathfrak a e}\in\mathbb Q_{>0}$ and $23$ is squarefree, its reduced denominator must be $1$. Thus $d_{\mathfrak a e} \in\mathbb Z_{\ge1}$, and hence  $d_1^{\mathrm{proj}}(E)=1$.

For every $\mathcal O_F$-sublattice $S\subset E$ of rank $2$, since $1\le\mu^p(S)$, Proposition~\ref{prop:hermitian_constant} yields
\[
1\le \mu^p(S) \le \gamma_{2,F} d_S^{1/2} = \sqrt{\frac{23}{5}}d_S^{1/2},
\]
and therefore $d_2^{\mathrm{proj}}(E)\ge5/23$.
On the other hand, when $\delta_4$ is given by~\eqref{eq:delta4},
\begin{align*}
u := -\zeta_5-\zeta_5^2+\alpha(1+\zeta_5+\zeta_5^3), \qquad
v := 2-\zeta_5^2+\zeta_5^3+\alpha\zeta_5
\end{align*}
satisfy
\[
h_4(u,u)=h_4(v,v)=1,
\qquad
h_4(u,v)
=
\frac{23+7\sqrt{-23}}{46},
\]
so the determinant of the Gram matrix with respect to $u$ and $v$ is $5/23$. Hence
$d_2^{\mathrm{proj}}(E)=5/23$.

Let $P\subset E$ be a saturated $\mathcal O_F$-sublattice of rank $3$ and put
$Q:=(F\otimes_{\mathcal O_F}P)^\perp\cap E^\#$.
Since $P\subset E$ splits as an $\mathcal O_F$-module, the canonical homomorphism
$E^\#\to P^\#$ is surjective, with kernel $Q$. Hence it induces an isomorphism
$E^\#/Q \cong P^\#$.
Moreover, if $E^\#/Q$ is endowed with the Hermitian form induced by orthogonal projection onto $F\otimes_{\mathcal O_F}P$, this isomorphism is an isometry.
Hence we have $d_E^{-1} = d_{E^\#} = d_{P^{\#}} d_Q = d_P^{-1}d_Q$.
Note that the Gram matrix in Remark~\ref{rem:E-Gram} has determinant $1/529$, so $d_E=1/529$.
Since $E^\#=\sqrt{-23}\,E$, we obtain
\[
d_P = \frac{d_Q}{529} \ge \frac{d_1^{\mathrm{proj}}(E^\#)}{529}
=
\frac{23}{529}=\frac1{23}.
\]
Since saturation can only decrease the discriminant, the same lower bound holds for every rank $3$ $\mathcal O_F$-sublattice of $E$.
The sublattice $\sum_{i=1}^{3}\mathcal O_Fp_i$ defined in Remark~\ref{rem:E-Gram} has discriminant $1/23$.
Thus $d_3^{\mathrm{proj}}(E)=1/23$.

Finally, every full-rank sublattice $M\subset E$ satisfies $d_M=[E:M]d_E\ge d_E$, with equality for $M=E$. Hence $d_4^{\mathrm{proj}}(E)=d_E=1/529$.
\end{proof}

\subsection{\texorpdfstring{The rank $11$ Hermitian lattice $J$}{The rank 11 Hermitian lattice J}}

First, note that since the narrow ideal class group of
$\mathbb{Q}(\zeta_{23})^+$ is trivial, every ideal of
$\mathbb{Q}(\zeta_{23})^+$ admits a totally positive generator.
We also note that the relative different
$\mathfrak D_{\mathbb{Q}(\zeta_{23})/F}$
is generated over $\mathcal O_{\mathbb{Q}(\zeta_{23})}$ by a totally positive element
$\delta_{11}\in \mathcal O_{\mathbb{Q}(\zeta_{23})^+}$.
For example, one may take
$\delta_{11}=(2-\zeta_{23}-\zeta_{23}^{-1})^5$.
Put
\[
J := (47, \zeta_{23}-21) \subset \mathcal O_{\Q(\zeta_{23})}
\]
and let $a$ be a totally positive generator of $J\cap\mathcal O_{\Q(\zeta_{23})^+}$.
Since $47\equiv1\pmod{23}$, the prime $47$ splits completely in $\Q(\zeta_{23})$, and hence $a\mathcal O_{\Q(\zeta_{23})}=J\bar J$.

\begin{definition}\label{def:rank11lattice}
Define a Hermitian form $h_{11}\colon J\times J\to\mathcal O_F$ by
\[
h_{11}(u,v) := \operatorname{Tr}_{\Q(\zeta_{23})/F}((\delta_{11}a)^{-1} u\bar v).
\]
\end{definition}

\begin{proposition}\label{prop:h11-properties}
The following statements hold.
\begin{enumerate}[label=(\roman*)]
\item $(J,h_{11})$ is a Hermitian lattice.
\item $h_{11}(J,J)\subset\mathcal O_F$.
\item $(J,h_{11})$ is self-dual.
\item For every $0\ne z\in J$, one has $h_{11}(z,z)\in\mathbb Z_{\geq 1}$. Hence the trace form $B_J(x,y):=\operatorname{Tr}_{F/\Q}h_{11}(x,y)$ is a positive definite even integral symmetric bilinear form. Moreover, $\det B_J=23^{11}$.

\end{enumerate}
\end{proposition}

\begin{proof}
Statement (i) follows from the total positivity of $a$ and $\delta_{11}$.
To prove (ii), note that
\[
(\delta_{11}a)^{-1}J\bar J=\delta_{11}^{-1}\mathcal O_{\Q(\zeta_{23})}=\mathfrak D_{\Q(\zeta_{23})/F}^{-1}.
\]
By the definition of the inverse different, it follows that $h_{11}(J,J)\subset\mathcal O_F$.
For (iii), the trace dual of a fractional ideal $I$ is $\mathfrak D_{\Q(\zeta_{23})/F}^{-1}I^{-1}$. Hence
\begin{align*}
J^\# =(\delta_{11}a)\,
\mathfrak D_{\Q(\zeta_{23})/F}^{-1}\,\bar J^{-1} =(\delta_{11}a)\,(\delta_{11})^{-1}\bar J^{-1} =a\bar J^{-1}=J.
\end{align*}
For every $0\ne z\in J$, statements (i) and (ii) give $h_{11}(z,z)\in\mathcal O_F\cap\Q_{>0}=\mathbb Z_{\geq 1}$.
Thus $B_J(z,z)=2h_{11}(z,z)\in2\mathbb Z$.
Let $J^*$ denote the $\mathbb Z$-dual lattice of $(J,B_J)$. By the definition of the inverse different and (iii),
\[
J^*= \mathfrak D_{F/\Q}^{-1}J^\# =\frac1{\sqrt{-23}}J.
\]
Since $J$ is a projective $\mathcal O_F$-module of rank $11$, we have $|J^*/J|=|\mathcal O_F/\mathfrak D_{F/\Q}|^{11}=23^{11}$. In particular, $\det B_J=23^{11}$.
\end{proof}

\begin{remark}\label{rem:norm-surj-11}
The group $\operatorname{N}_{\Q(\zeta_{23})/\Q(\zeta_{23})^+}(\mathcal O_{\Q(\zeta_{23})}^\times)$ coincides with the group of totally positive units of $\mathcal O_{\Q(\zeta_{23})^+}$. Therefore the isometry class of the Hermitian lattice $(J,h_{11})$ is independent of the choices of $\delta_{11}$ and $a$.
\end{remark}


\subsection{\texorpdfstring{Definition of the rank $88$ even unimodular lattice}{Definition of the rank 88 even unimodular lattice}}
\label{sec:integral-normalization}

\begin{definition}
Define $(L,h):=(E,h_4)\otimes_{\mathcal O_F}(J,h_{11})$. This is a Hermitian lattice of rank $44$.
\end{definition}

\begin{remark}
Since the discriminants of $F(\zeta_5)/F$ and $\Q(\zeta_{23})/F$ are relatively prime, there is a natural isomorphism of $\mathcal O_F$-algebras
\[
\mathcal O_{F(\zeta_5)}\otimes_{\mathcal O_F}\mathcal O_{\Q(\zeta_{23})} \stackrel{\sim}{\to}  \mathcal O_{\Q(\zeta_{115})}; \quad x \otimes y \mapsto xy
\]
(cf.~\cite[Proposition~2.11]{Neukirch1999}).
This isomorphism induces $L\stackrel{\sim}{\to}J\mathcal O_{\Q(\zeta_{115})}$, and $h=h_4\otimes h_{11}$ corresponds to the form $(z,w)\mapsto\operatorname{Tr}_{\Q(\zeta_{115})/F}\!\left((\delta_4\delta_{11}a)^{-1}z\bar w\right)$.
Since $\delta_4\delta_{11}$ is a generator of $\mathfrak D_{\Q(\zeta_{115})/\Q}$,
the lattice $(L,h)$ is isomorphic to the ideal lattice associated with $J\mathcal O_{\Q(\zeta_{115})}$.
\end{remark}

\begin{proposition}\label{prop:tensor-decomposition}
The following statements hold.
\begin{enumerate}[label=(\roman*)]

\item $h(L,L)\subset\mathfrak D_{F/\Q}^{-1}$ and $L^\#=\sqrt{-23}\,L$. In particular, $h(z,z) \in \mathbb{Z}_{\ge 1}$ for any $0 \neq z \in L$.

\item Put $B(x,y):=\operatorname{Tr}_{F/\Q}h(x,y)$. Then $(L,B)$ is a positive definite even unimodular lattice of rank $88$.
\end{enumerate}
\end{proposition}

\begin{proof}
For (i), Propositions~\ref{prop:rank4-algebraic-properties}(ii) and \ref{prop:h11-properties}(ii) give
\[
h_4(E,E)\subset\mathfrak D_{F/\Q}^{-1},
\qquad
h_{11}(J,J)\subset\mathcal O_F,
\]
and hence $(h_4\otimes h_{11})(L,L)\subset\mathfrak D_{F/\Q}^{-1}$. Moreover, Propositions \ref{prop:rank4-algebraic-properties}(iii) and \ref{prop:h11-properties}(iii) imply
\[
L^\#
=(E\otimes_{\mathcal O_F}J)^\#
= E^\#\otimes_{\mathcal O_F}J^\#
=(\sqrt{-23}\,E)\otimes_{\mathcal O_F}J
=\sqrt{-23}\,L.
\]
We next prove (ii).
Since $h(z,z)\in\mathbb Z_{\ge1}$ for every nonzero $z$, we have $B(z,z)=2h(z,z)\in2\mathbb Z$. Hence $B$ is even.
Furthermore, by the definition of the inverse different and (i),
\begin{align*}
\{z\in L\otimes_{\mathcal O_F}F \mid B(z,L)\subset\mathbb Z\}
=\{z \mid h(z,L)\subset\mathfrak D_{F/\Q}^{-1}\} = \{z \mid \sqrt{-23}\,z\in L^\#\} = L,
\end{align*}
and therefore $B$ is unimodular.
\end{proof}

The following theorem is the main result of this section. Its proof is given in \S\ref{sec:rank4-audit}.

\begin{theorem}\label{thm:main}
For every $0\ne z\in L$, one has $h(z,z)\ge4$.
\end{theorem}

\begin{corollary}\label{cor:extremal}
The lattice $(L,B)$ is an extremal positive definite even unimodular lattice of rank $88$.
\end{corollary}
\begin{proof}
For $z\in L$, we have $B(z,z)=2h(z,z)$. Thus Theorem~\ref{thm:main}, together with $2\left\lfloor88/24\right\rfloor+2=8$, shows that $(L,B)$ is  extremal.
\end{proof}

\begin{remark}\label{rem:description}
We describe the analogy between Nebe's lattice \(\Gamma_{72}\) and the lattice \((L,h)\).
In \cite[Theorem~4.5]{Nebe2016}, Nebe proved that $\Gamma_{72}$ is represented by an ideal lattice
$(I,b_\alpha)$ for each of the six ideal classes $[I]$ of order $7$ in
$\operatorname{Cl}(\mathbb{Q}(\zeta_{91}))$.
Since
\[
\operatorname{Cl}\bigl(\mathbb{Q}(\sqrt{-7},\zeta_{13})\bigr)\cong C_7 \qquad
\text{and} \qquad
[\mathbb{Q}(\zeta_{91}):\mathbb{Q}(\sqrt{-7},\zeta_{13})]=3,
\]
the canonical homomorphism
$\operatorname{Cl}\bigl(\mathbb{Q}(\sqrt{-7},\zeta_{13})\bigr)
\hookrightarrow
\operatorname{Cl}\bigl(\mathbb{Q}(\zeta_{91})\bigr)$
is injective. Its image is therefore the subgroup of order $7$ in
$\operatorname{Cl}(\mathbb{Q}(\zeta_{91}))$.
Consequently, the ideal defining $\Gamma_{72}$ can be taken to be the extension of an ideal $I'$ of $\mathbb{Q}(\sqrt{-7},\zeta_{13})$, and hence $\Gamma_{72}$ admits a description of the form
\[
I' \otimes_{\mathcal{O}_{\mathbb{Q}(\sqrt{-7})}} \mathcal{O}_{\mathbb{Q}(\zeta_7)}.
\]
Choose a totally positive generator in
$\mathcal O_{\mathbb Q(\zeta_7)^+}$ of the relative different
$\mathfrak D_{\mathbb Q(\zeta_7)/\mathbb Q(\sqrt{-7})}$ and, as in
Definition~\ref{def:rank11lattice}, use it to endow $\mathcal O_{\mathbb Q(\zeta_7)}$ with the structure of a Hermitian lattice over $\mathcal O_{\mathbb Q(\sqrt{-7})}$.
It is easily verified that the resulting Hermitian lattice is isometric to the Barnes lattice $P_b$.
Since $L =  \mathcal{O}_{F(\zeta_5)} \otimes_{\mathcal{O}_F}J$, the lattice $L$ may, in this sense, be regarded as an $88$-dimensional analogue of $\Gamma_{72}$.
\end{remark}

\section{\texorpdfstring{The rank $11$ Hermitian lattice $(J,h_{11})$}{The rank 11 Hermitian lattice (J,h11)}}\label{sec:d1d2}

In this section, we study the structure of the rank \(11\) Hermitian lattice \((J,h_{11})\) over \(\mathcal O_F\) and compute \(d_1^{\mathrm{proj}}(J)\), \(d_2^{\mathrm{proj}}(J)\), \(d_3^{\mathrm{proj}}(J)\), and \(d_4^{\mathrm{proj}}(J)\). The computations used in this section are available in the  GitHub repository. 
For the computations, one must make explicit choices of $\delta_{11}$ and $a$; in this paper we take
\begin{align*}
\delta_{11} &=  (2-\zeta_{23}-\zeta_{23}^{-1})^5,
\\
a &=13
-(\zeta_{23}+\zeta_{23}^{-1})
-4(\zeta_{23}^{2}+\zeta_{23}^{-2})
-(\zeta_{23}^{3}+\zeta_{23}^{-3})
+6(\zeta_{23}^{4}+\zeta_{23}^{-4})
+5(\zeta_{23}^{5}+\zeta_{23}^{-5})
\\
&\quad -5(\zeta_{23}^{6}+\zeta_{23}^{-6})
+4(\zeta_{23}^{8}+\zeta_{23}^{-8})
+3(\zeta_{23}^{9}+\zeta_{23}^{-9})
+(\zeta_{23}^{10}+\zeta_{23}^{-10}).
\end{align*}
Also note that $\mathrm{Cl}(F)\cong C_3$, with representatives
\[
\mathcal O_F, \qquad \mathfrak p := \left(2,\alpha\right), \qquad \bar{\mathfrak p} := \left(2,\bar\alpha\right).
\]
For every $\mathcal O_F$-sublattice $S\subset J$, the integrality of
$h_{11}$ implies $d_S\in\mathbb Z_{>0}$.
Moreover, saturation can only decrease the discriminant, so it suffices
to consider saturated sublattices when proving projective lower bounds.

\subsection{\texorpdfstring{$d_1^{\mathrm{proj}}(J) = 6$}{d1 projective(J) = 6}}

Every rank $1$ projective sublattice of $J$ can be written as $\mathfrak a e\subset J$, with $\mathfrak{a} \in \{ \mathcal{O}_F, \mathfrak{p}, \overline{\mathfrak{p}} \}$ and $e\in\mathfrak a^{-1}J$, and  $d_{\mathfrak a e} = \operatorname{N}_{F/\mathbb Q}(\mathfrak a)h_{11}(e,e)$.
An exact enumeration of all nonzero vectors up to the required norm bounds in the three lattices $\mathfrak a^{-1}J$, with the Hermitian forms $\operatorname{N}_{F/\mathbb Q}(\mathfrak a)h_{11}$, gives the following proposition.

\begin{proposition}\label{prop:d_1(J)}
We have $d_1^{\mathrm{proj}}(J)=6$.
\end{proposition}

\subsection{\texorpdfstring{$\Aut_{\mathcal O_F}(J,h_{11})$}{Automorphism group of (J,h11)}}\label{sec:autgroup}

Define the automorphism group of $(J,h_{11})$ by
\[
\Aut_{\mathcal O_F}(J,h_{11}) := \{g\in\GL_{\mathcal O_F}(J) \mid h_{11}(gx,gy)=h_{11}(x,y)\text{ for all }x,y\in J\}.
\]
Computation shows that $\Aut_{\mathbb Z}(J,B_J)$ acts faithfully and transitively on $\{x\in J\mid B_J(x,x)=12\}$. The stabilizer has order $12$, and
\[
\Aut_{\mathcal O_F}(J,h_{11}) = \Aut_{\Z}(J,B_J) \cong C_2\times\PSL_2(23)
\]
(see Proposition~\ref{grp:prop:Jgroups}). 
The following tables give the numbers of $\Aut_{\mathcal O_F}(J,h_{11})$-orbits of saturated rank $1$ $\mathcal O_F$-sublattices $P\subset J$, classified by discriminant and Steinitz class:
\[
\begin{array}{c|rrr}
d_{P}&
[\mathcal O_F]&
[\mathfrak p]&
[\bar{\mathfrak p}]\\
\hline
6 &1&1&1\\
8 &3&4&3\\
9 &3&3&3\\
10&12&9&12\\
11&14&14&14
\end{array}
\qquad
\begin{array}{c|rrr}
d_{P}&
[\mathcal O_F]&
[\mathfrak p]&
[\bar{\mathfrak p}]\\
\hline
12&43&45&43\\
13&62&68&62\\
14&158&160&158\\
15&268&264&268\\
16&533&541&533
\end{array}
\qquad
\begin{array}{c|rrr}
d_{P}&
[\mathcal O_F]&
[\mathfrak p]&
[\bar{\mathfrak p}]\\
\hline
17&920&903&920\\
18&1693&1661&1693\\
19&2749&2748&2749\\
20&4667&4722&4667\\
21&7405&7499&7405
\end{array}
\]
The equality, throughout the displayed range, of the orbit counts for the Steinitz classes $[\mathcal O_F]$ and $[\bar{\mathfrak p}]$ is explained by a semilinear similitude of $J$.

First, as an element of $\Cl(\mathcal O_{\Q(\zeta_{23})})\cong C_3$, one has $[J]=[\mathfrak p\mathcal O_{\Q(\zeta_{23})}]$.
Since $J\bar J$ is principal, we obtain $[\bar{\mathfrak p}^{-1}J\bar J^{-1}]=[\mathfrak p\mathcal O_{\Q(\zeta_{23})}][J]^2=[\mathfrak p\mathcal O_{\Q(\zeta_{23})}]^3=1$.
Hence there exists $\beta\in\Q(\zeta_{23})^\times$ such that $\beta\bar J=\bar{\mathfrak p}^{-1}J$.
Multiplying by the complex conjugate gives
\[
(\beta\bar\beta) =
(\mathfrak p\bar{\mathfrak p}\mathcal O_{\Q(\zeta_{23})})^{-1} =  \left(2^{-1}\right).
\]
Thus $2\beta\bar\beta$ is a totally positive unit in $\mathcal O_{\Q(\zeta_{23})^+}$.
Moreover, since the unit signature rank of $\Q(\zeta_{23})^+$ is maximal,
there exists $v\in\mathcal O_{\Q(\zeta_{23})^+}^{\times}$ such that $2\beta\bar\beta=v^2$.
Replacing $\beta$ by $\beta/v$, we may arrange that $2\beta\bar\beta=1$.
Now define
\[
T(x) := \beta\bar x.
\]
Then $T(J)=\bar{\mathfrak p}^{-1}J$ and $T(cx)=\bar c\,T(x)$ for $c\in F$.
Moreover,
\begin{align*}
h_{11}(Tx,Ty)
=
\Tr_{\Q(\zeta_{23})/F}
\bigl((\delta_{11}a)^{-1}
\beta\bar\beta\,\bar x y\bigr) =
\frac12
\Tr_{\Q(\zeta_{23})/F}
\bigl((\delta_{11}a)^{-1}\bar x y\bigr) =
\frac12\,\overline{h_{11}(x,y)}.
\end{align*}
Let $P=\mathfrak a e\subset J$ be an $\mathcal O_F$-sublattice of rank 1.
Then $P^\dagger:=\bar{\mathfrak p}\,T(P)=\bar{\mathfrak p}\,\bar{\mathfrak a}\,T(e)\subset J$ is again an $\mathcal O_F$-sublattice  of rank 1, and
$d_{P^\dagger}=2 \operatorname{N}_{F/\Q}(\mathfrak a)h_{11}(T(e),T(e))=d_{P}$.
Therefore $P\mapsto P^\dagger$ is a discriminant-preserving
involution on the $\mathcal O_F$-sublattices of rank $1$.
Moreover, for any $g\in \Aut_{\mathcal O_F}(J,h_{11})$, one has $(gP)^\dagger=(TgT^{-1})P^\dagger$, so the numbers of $\Aut_{\mathcal O_F}(J,h_{11})$-orbits are also preserved.

\subsection{\texorpdfstring{$d_2^{\mathrm{proj}}(J) = 16$}{d2 projective(J) = 16}}

\begin{proposition}\label{prop:d_2(J)}
We have $d_2^{\mathrm{proj}}(J)=16$.
\end{proposition}

\begin{remark}
The proofs of Propositions~\ref{prop:d_2(J)},~\ref{prop:d_3(J)}, and~\ref{prop:d_4(J)} use the same computational approach, based on the proof of~\cite[Proposition~5.5(c)]{CoulangeonNebe2013}.
\end{remark}

\begin{proof}
Suppose that there exists an $\mathcal O_F$-sublattice $S\subset J$ of  rank $2$ with $d_S\le15$.
By Proposition~\ref{prop:hermitian_constant}, $\gamma_{2,F}^{(c)}=\sqrt{23/5}$ for every Steinitz class $c$, so $\mu^p(S)\le\sqrt{23/5}\sqrt{15}=\sqrt{69}$. Hence there exists an  $\mathcal O_F$-sublattice $P_0\subset S$ of  rank $1$ with $d_{P_0}\le8$.
If $P_0\ne P_0^{\mathrm{sat}}$, then $d_{P_0}\ge2 d_{P_0^{\mathrm{sat}}}\ge12$, a contradiction. Thus $P_0$ is saturated.
There are exactly $13$ $\Aut_{\mathcal O_F}(J,h_{11})$-orbits of such $P_0$.
Since $P_0$ is saturated, the quotient $S/P_0$ is projective. Thus, if $P_0=\mathfrak a e$, we may write
$S=\mathfrak a e\oplus\mathfrak b f$, where
$\mathfrak b\in\{\mathcal O_F,\mathfrak p,\bar{\mathfrak p}\}$ and
$f\in\mathfrak b^{-1}J$.
For each representative $P_0=\mathfrak a e$ of the $13$ orbits, put
$D_e(f):=h_{11}(e,e)h_{11}(f,f)-\Nm_{F/\Q}(h_{11}(e,f))$. Then
\[
d_S = \operatorname{N}_{F/\Q}(\mathfrak a)\operatorname{N}_{F/\Q}(\mathfrak b)D_e(f).
\]
We examine all three choices of $\mathfrak b$, using the induced
positive definite form on the quotient by $Fe$.
Exact enumeration finds no $f$ satisfying
$0<\operatorname{N}_{F/\Q}(\mathfrak a)\operatorname{N}_{F/\Q}(\mathfrak b)D_e(f)\le 15$.
Hence $d_2^{\mathrm{proj}}(J)\ge16$.
Conversely, there exists an $\mathcal{O}_F$-sublattice $S=\mathfrak p e\oplus\mathfrak p^{-1}f\subset J$ such that
\[
h_{11}(e,e)=4,
\qquad
h_{11}(f,f)=120,
\qquad
h_{11}(e,f)=-21+\sqrt{-23}.
\]
For this sublattice, $d_S=4\cdot120-464=16$.
Therefore, $d_2^{\mathrm{proj}}(J) = 16$.
\end{proof}

\subsection{\texorpdfstring{$d_3^{\mathrm{proj}}(J) = 27$}{d3 projective(J) = 27}}
\label{sec:d3-computation}

\begin{proposition}\label{prop:d_3(J)}
We have $d_3^{\mathrm{proj}}(J) =27$.
\end{proposition}
\begin{proof}
The lattice $J$ contains a free $\mathcal O_F$-sublattice of rank $3$  with Gram matrix
\[
3\begin{pmatrix}
2&-1&-\alpha\\
-1&3&1\\
-\bar\alpha&1&4
\end{pmatrix}.
\]
Hence $d_3^{\mathrm{proj}}(J)  \le27$.
We verify the projective lower bound computationally.
Suppose that a saturated $\mathcal O_F$-sublattice $S\subset J$  of rank $3$ satisfies
$d_S\le26$. By Proposition~\ref{prop:hermitian_constant},
\[
(\gamma_{3,F}^{(c)})^3\le46,
\qquad (\gamma_{2,F}^{(c)})^2=\frac{23}{5}
\]
for every Steinitz class $c$. Hence
\[
\mu^p(S)\le(46d_S)^{1/3}\le(46\cdot26)^{1/3}<11.
\]
Thus $S$ contains a rank $1$ sublattice $P_0$ of discriminant
$d_{P_0}\le10$.
Such a $P_0$ is saturated in $J$. The tables in \S\ref{sec:autgroup} give $55$ $\Aut_{\mathcal O_F}(J,h_{11})$-orbits of such $P_0$.

The quotient $S/P_0$, equipped with the orthogonal quotient form, has
rank $2$ and discriminant $d_S/d_{P_0}$.
Since $(\gamma_{2,F}^{(c)})^2=23/5$, there exists a saturated rank $2$ sublattice $R$ with $P_0\subset R\subset S$ such that
\[
d_R\le \left\lfloor\sqrt{\frac{23d_Sd_{P_0}}{5}}\right\rfloor
\le \left\lfloor\sqrt{\frac{598d_{P_0}}{5}}\right\rfloor.
\]
We therefore enumerate, for each possible $P_0$, all saturated rank $2$ sublattices $R$ satisfying this bound.
This yields $978$ distinct sublattices $R$.

For each such $R$, let $Q_R:=J/R$ be the projective $\mathcal O_F$-module of rank $9$. We endow $Q_R$ with the Hermitian form $h_{Q_R}$ induced by the orthogonal projection $F \otimes_{\mathcal{O}_F}J \to (F \otimes_{\mathcal{O}_F}R)^\perp$.
For each $\mathfrak a\in\{\mathcal O_F,\mathfrak p,\overline{\mathfrak p}\}$, a nonzero element $e\in\mathfrak a^{-1}Q_R$ determines a rank $1$ projective submodule $\mathfrak a e\subset Q_R$, and its inverse image in $J$ is a rank $3$ sublattice $S$ containing $R$. Its discriminant is $d_S=d_R\operatorname{N}_{F/\mathbb Q}(\mathfrak a)h_{Q_R}(e,e)$. Thus, for each $R$ and each $\mathfrak a\in\{\mathcal O_F,\mathfrak p,\overline{\mathfrak p}\}$, it suffices to enumerate the elements $e\in\mathfrak a^{-1}Q_R$ satisfying $0 < d_R\operatorname{N}_{F/\mathbb Q}(\mathfrak a)h_{Q_R}(e,e)\le26$.
None of these enumerations produces a saturated $\mathcal{O}_F$-sublattice $S$ of rank $3$ with $d_S\le26$. Therefore $d_3^{\mathrm{proj}}(J)\ge27$.
\end{proof}

\subsection{\texorpdfstring{$d_4^{\mathrm{proj}}(J)= 48$}{d4 projective(J) = 48}}
\begin{proposition}\label{prop:d_4(J)}
We have $d_4^{\mathrm{proj}}(J)=48$.
\end{proposition}

\begin{proof}
For a rank-$4$ sublattice $S\subset J$ with $d_S\le47$, the Hermitian bound $\gamma_8=2$ gives a nonzero vector $x\in S$ with $h_{11}(x,x)\le\sqrt{23}\,d_S^{1/4}<13$.
Using the same method as in Proposition~\ref{prop:d_3(J)}, we find no such \(S\). Hence \(d_4^{\mathrm{proj}}(J)\ge48\). 
On the other hand, $J$ contains a free $\mathcal O_F$-sublattice of rank $4$ with Gram matrix
\[
2\begin{pmatrix}
3&-1&\alpha-1&-1-\alpha\\
-1&3&0&1\\
-\alpha&0&4&-1\\
\alpha-2&1&-1&5
\end{pmatrix},
\]
whose discriminant is $48$. Therefore $d_4^{\mathrm{proj}}(J)=48$.
\end{proof}

\section{\texorpdfstring{Proof of Theorem~\ref{thm:main}}{Proof of the minimum theorem}}

\subsection{\texorpdfstring{The case of tensor rank at most $3$}{The case of tensor rank at most 3}}

\begin{proposition}\label{prop:tensor-rank-3}
If $0\ne z\in L=E\otimes_{\mathcal O_F}J$ has tensor rank at most $3$, then $h(z,z)\ge4$.
\end{proposition}
\begin{proof}
Let $1\le s\le3$ be the tensor rank of $z$. By Proposition~\ref{prop:inequality-tensor},
\[
h(z,z)\ge s\left(d_s^{\mathrm{proj}}(E)d_s^{\mathrm{proj}}(J)\right)^{1/s}.
\]
Propositions~\ref{prop:d_s(E)}, \ref{prop:d_1(J)}, \ref{prop:d_2(J)}, and~\ref{prop:d_3(J)} then imply that $h(z,z)>3$.
Since $h(z,z)\in\mathbb Z$, we obtain $h(z,z)\ge4$.
\end{proof}

\subsection{\texorpdfstring{The case of tensor rank $4$}{The case of tensor rank 4}}

\begin{proposition}\label{prop:tensor-rank-4-case1}
If $z\in L$ has tensor rank $4$, then $h(z,z)\ge3$.
\end{proposition}
\begin{proof}
Proposition~\ref{prop:inequality-tensor}, together with Propositions~\ref{prop:d_s(E)} and~\ref{prop:d_4(J)},  gives $h(z,z)\ge4(48/529)^{1/4}>2$.
Since $h(z,z)\in\mathbb Z$, we obtain $h(z,z)\ge3$.
\end{proof}

\begin{proposition}\label{prop:tensor-rank-4-case2}
If $z\in L$ has tensor rank $4$, then $h(z,z)\ne3$.
\end{proposition}
\begin{proof}
We use the same notation as in Remark \ref{rem:E-Gram}.
Suppose that $z\in L = E\otimes_{\mathcal O_F}J$ has tensor rank $4$ and satisfies $h(z,z)=3$. Then one can write
\[
z
=
p_1\otimes x_1
+p_2\otimes x_2
+p_3\otimes x_3
+p_4\otimes x_4,
\qquad
x_1,x_2,x_3,x_4\in J.
\]
Since $z$ has tensor rank $4$, the vectors $x_1,x_2,x_3,x_4$ are linearly independent over $F$. Put
\[
\begin{aligned}
y_1 :=\sqrt{-23}\,x_1-4x_3+(\alpha-1)x_4, \qquad
y_2 :=\sqrt{-23}\,x_2+\alpha x_3-4x_4, \qquad
y_3 :=x_3, \qquad
y_4 :=x_4.
\end{aligned}
\]
Then $y_1,y_2,y_3,y_4\in J$ and
\[
23h(z,z)
=
q(y_1)+q(y_2)+q(y_3)+q(y_4), \qquad q(x) := h_{11}(x,x).
\]
Moreover, $y_1,y_2,y_3,y_4$ are linearly independent over $F$.
Let the change-of-basis matrix carrying ${}^{t}(x_1,x_2,x_3,x_4)$ to ${}^{t}(y_1,y_2,y_3,y_4)$ be
\[
T :=
\begin{pmatrix}
\sqrt{-23}&0&-4&\alpha-1\\
0&\sqrt{-23}&\alpha&-4\\
0&0&1&0\\
0&0&0&1
\end{pmatrix}.
\] Put
\[
H_X := \bigl(h_{11}(x_i,x_j)\bigr)_{1\le i,j\le4},
\qquad
H_Y := \bigl(h_{11}(y_i,y_j)\bigr)_{1\le i,j\le4}.
\]
Then $H_Y=TH_X\overline{T}^{\mathsf T}$.
It follows from $h(z,z)=3$ that
\[
\operatorname{tr}(H_Y) = q(y_1)+q(y_2)+q(y_3)+q(y_4) = 69.
\]
By Proposition~\ref{prop:d_1(J)}, we have $d_1^{\mathrm{proj}}(J)=6$, and hence $q(y_i)\ge6$. Thus the diagonal entries of $H_Y$ satisfy
\[
q(y_i)\in\mathbb Z_{\ge6}, \qquad q(y_1)+q(y_2)+q(y_3)+q(y_4)=69,
\]
and there are only finitely many such tuples.
There are also only finitely many possibilities for the off-diagonal entries.
Indeed, since $H_Y$ is positive definite, every $2\times2$ principal minor is positive. Therefore
\[
\operatorname{N}_{F/\mathbb Q}\bigl(h_{11}(y_i,y_j)\bigr)
<
q(y_i)q(y_j)
\qquad (i\ne j).
\]
Since $q(y_i)q(y_j)$ takes only finitely many values, so does $\operatorname{N}_{F/\mathbb Q}(h_{11}(y_i,y_j))$.
Since $h_{11}(y_i,y_j)\in\mathcal O_F$ and $\mathcal{O}_F$ is a lattice in $\mathbb{C}$, the above bound leaves only finitely many possibilities for $h_{11}(y_i,y_j)$.

Thus there are only finitely many candidates for the positive definite Hermitian matrix $H_Y$, and all of them can be enumerated.
Moreover, the arithmetic--geometric mean inequality gives $\det H_Y\le(69/4)^4$, so from $H_Y=TH_X\overline{T}^{\mathsf T}$ and Proposition~\ref{prop:d_4(J)} we obtain $48\le\det H_X\le167$.
Finally, using elements of $\operatorname{Aut}_{\mathcal O_F}(E,h_4)$, we may arrange that $6\le q(y_4)\le q(y_3)$ and $12\le q(y_3)+q(y_4)\le34$.
The exact enumeration leaves $8159$ candidates. Details are provided in the computational supplement. 

For each candidate $H_X$, let $S$ be a free Hermitian $\mathcal O_F$-lattice with  Gram matrix $H_X$, and write $q(x)$ for its Hermitian norm.
Enumerate the vectors of Hermitian norm at most $12$ in $S$ and in $J$. For pairs $u,v\in S$ with $q(u),q(v)\le12$, test whether the sublattice $\mathcal O_Fu+\mathcal O_Fv$ admits an isometric embedding into $J$. These tests reduce the number of candidates to $1376$.
Identifying candidates that are equivalent under $\operatorname{Aut}_{\mathcal O_F}(E,h_4)$ leaves $362$ candidates.
For each of these $362$ candidates $S$ whose vectors of norm at most $12$ span $F\otimes_{\mathcal O_F}S$, choose four $F$-linearly independent such vectors $u_1,u_2,u_3,u_4$ and test whether $\sum_{i=1}^4\mathcal O_Fu_i$ admits an isometric embedding into $J$.
If such embeddings exist, test whether any of them extends to an isometric embedding $S\hookrightarrow J$.
Then only one candidate remains, for which $\dim_F\Span_F\{x\in S\mid q(x)\le12\}=2$.
For this final candidate, we may choose an $\mathcal O_F$-basis $u_1,u_2,u_3,u_4$ of $S$ such that $(q(u_1),q(u_2),q(u_3),q(u_4))=(12,14,13,13)$. A further  computation shows that the Hermitian lattice $S=\sum_{i=1}^4\mathcal O_Fu_i$ does not admit an isometric embedding into $J$.
It follows that there is no element $z\in L$ of tensor rank $4$ satisfying $h(z,z)=3$.
The complete sources, exact input forms, embedding witnesses, independent verification, and measured timings are included in the computational supplement.
\end{proof}

\subsection{\texorpdfstring{Proof of Theorem~\ref{thm:main}}{Completion of the minimum proof}}\label{sec:rank4-audit}

Since $L=E\otimes_{\mathcal O_F}J$ and $E$ is a free $\mathcal O_F$-module of rank $4$, every $0\ne z\in L$ has tensor rank at most $4$.
Therefore Propositions~\ref{prop:tensor-rank-3},~\ref{prop:tensor-rank-4-case1}, and~\ref{prop:tensor-rank-4-case2} imply that $h(z,z)\ge4$.

\subsection{\texorpdfstring{A second extremal lattice of rank $88$}{A second extremal lattice of rank 88}}
\label{sec:J0-extremal}

We now describe a second extremal lattice obtained from the principal ideal
class of $\Q(\zeta_{23})$.
Put 
\[
J_0:=\bar{\mathfrak p}J\subset\mathcal O_{\mathbb{Q}(\zeta_{23})},
\]
where $\mathfrak p=(2,\alpha) \subset \mathbb{Z}[\alpha]$.
Note that $J_0\bar J_0 = 2a\mathcal{O}_{\mathbb{Q}(\zeta_{23})}$.
We therefore equip $J_0$ with the Hermitian form
\[
h_{11}^{(0)}(u,v)
:=
\operatorname{Tr}_{\Q(\zeta_{23})/F}
\left((2\delta_{11}a)^{-1}u\bar v\right).
\]
Thus, on $J_0\subset J$, one has $2h_{11}^{(0)} = h_{11}$.
The same argument as in Proposition~\ref{prop:h11-properties} shows that
$(J_0,h_{11}^{(0)})$ is a self-dual Hermitian lattice over $\mathcal O_F$.
Moreover, $J_0$ belongs to the principal ideal class of $\mathbb{Q}(\zeta_{23})$ since $J_0$ is generated by $1+\zeta_{23}+\zeta_{23}^4+\zeta_{23}^6+
\zeta_{23}^8+\zeta_{23}^9+\zeta_{23}^{12}+\zeta_{23}^{20}$.

\begin{proposition}\label{prop:dproj-J0}
For every $1\le s\le11$, one has  $d_s^{\mathrm{proj}}(J_0) = d_s^{\mathrm{proj}}(J)$.
\end{proposition}

\begin{proof}
Let $P=\mathfrak a_1e_1\oplus\cdots\oplus\mathfrak a_se_s \subset J$
be a projective $\mathcal O_F$-sublattice of rank $s$.
Then  $\bar{\mathfrak p}P =
(\bar{\mathfrak p}\mathfrak a_1)e_1
\oplus\cdots\oplus
(\bar{\mathfrak p}\mathfrak a_s)e_s
\subset J_0$.
Since $\operatorname{N}_{F/\Q}(\bar{\mathfrak p})=2$ and
$2h_{11}^{(0)} = h_{11}$, we obtain
\begin{align*}
d_{\bar{\mathfrak p}P}
&=
\left(
\prod_{i=1}^s
\operatorname{N}_{F/\Q}(\bar{\mathfrak p}\mathfrak a_i)
\right)
\det\left(
h_{11}^{(0)}(e_i,e_j)
\right)_{i,j}
\\
&=
2^s
\left(
\prod_{i=1}^s
\operatorname{N}_{F/\Q}(\mathfrak a_i)
\right)
2^{-s}
\det\left(
h_{11}(e_i,e_j)
\right)_{i,j}
\\
&=
d_P.
\end{align*}
The correspondence $P\mapsto\bar{\mathfrak p}P$ is a bijection between projective rank-$s$ sublattices of $J$ and those
of $J_0$, with inverse $P_0\mapsto\bar{\mathfrak p}^{-1}P_0$.
The assertion follows.
\end{proof}

Define
\[
(L_0,h^{(0)})
:=
(E,h_4)\otimes_{\mathcal O_F}(J_0,h_{11}^{(0)}),
\qquad
B_0(x,y):=
\operatorname{Tr}_{F/\Q}h^{(0)}(x,y).
\]
By exactly the same argument as in Proposition~\ref{prop:tensor-decomposition},
$(L_0,B_0)$ is a positive definite even unimodular lattice of rank $88$.
Moreover, $(L_0,B_0)$ is not isomorphic to $(L,B)$ as a $\mathbb Z$-lattice.
Indeed, Theorem~\ref{grp:thm:fullgroups} in Appendix~\ref{app:fullgroups} gives
\begin{align*}
\Aut_{\mathbb Z}(L,B)
&\cong
\Aut_{\mathcal O_F}(E,h_4)\times\PSL_2(23),
\\
\Aut_{\mathbb Z}(L_0,B_0)
&\cong
\bigl(\Aut_{\mathcal O_F}(E,h_4)\times\PSL_2(23)\bigr)\rtimes C_2.
\end{align*}

\begin{proposition}\label{prop:L0-extremal}
The positive definite even unimodular lattice $(L_0,B_0)$ is extremal.
\end{proposition}

\begin{proof}
By Proposition~\ref{prop:dproj-J0}, the arguments of Propositions~\ref{prop:tensor-rank-3} and~\ref{prop:tensor-rank-4-case1} apply to $L_0$.
Repeating the enumeration of Proposition~\ref{prop:tensor-rank-4-case2} for $J_0$ excludes Hermitian norm $3$.
Thus $(L_0,B_0)$ is extremal.
\end{proof}

\begin{remark}\label{rem:ideal-class-exhaustion}
We finally note that allowing an arbitrary ideal of $\Q(\zeta_{23})$ in place of $J$ does not yield any further isometry classes of $\mathbb Z$-lattices.
More precisely, let $I$ be a nonzero ideal of $\mathcal O_{\Q(\zeta_{23})}$. Since the narrow class group of $\Q(\zeta_{23})^+$ is trivial, we may choose a totally positive element $a_I\in\Q(\zeta_{23})^+$ such that $I\bar I=a_I\mathcal O_{\Q(\zeta_{23})}$, and equip $I$ with the Hermitian form
$h_I(u,v)=\operatorname{Tr}_{\Q(\zeta_{23})/F}\bigl((\delta_{11}a_I)^{-1}u\bar v\bigr)$.
If $I'=xI$ for some $x\in\Q(\zeta_{23})^\times$, then we may take $a_{I'}=x\bar x\,a_I$, and multiplication by $x$ induces an isometry $(I,h_I)\cong(I',h_{I'})$. 
It follows from Remark \ref{rem:norm-surj-11} that changing the totally positive generator $a_I$ does not change the isometry class. 
Thus the isometry class of the normalized ideal lattice depends only on the ideal class of $I$.

Since $\Cl(\mathcal O_{\Q(\zeta_{23})})\cong C_3$ and its two nontrivial classes are represented by $J$ and $\bar J$, there are only three ideal classes to consider: the principal class, represented by $J_0$, and the two conjugate nontrivial classes represented by $J$ and $\bar J$.
Moreover, complex conjugation induces a $\mathbb Z$-isometry
$E\otimes_{\mathcal O_F}J\to E\otimes_{\mathcal O_F}\bar J$,
given by $u\otimes x\mapsto\bar u\otimes\bar x$, between the corresponding trace lattices. Hence the two nontrivial ideal classes give isometric $\mathbb Z$-lattices. Consequently, the normalized ideal-lattice construction over $\Q(\zeta_{23})$ yields, up to $\mathbb Z$-isometry, precisely the two lattices $(L,B)$ and $(L_0,B_0)$.
\end{remark}

\section{On Extremeness}

In this section, we recall the definitions of perfect, eutactic, and extreme lattices and show that both $(L,B)$ and $(L_0,B_0)$ are extreme.

For a positive definite $\mathbb Z$-lattice $M$ of rank $r$, put
\[
\min(M) := \min_{x \in M\setminus\{0\}}  (x,x), \qquad
\operatorname{Min}(M) := \{x\in M\setminus\{0\}\mid (x,x)=\min(M)\}.
\]
The lattice $M$ is called \emph{perfect} if
\[
\operatorname{span}_{\mathbb R}
\{xx^{\mathsf T} \mid x\in\operatorname{Min}(M)\}
=
\operatorname{Sym}_r(\mathbb R).
\]
Here we choose an orthonormal basis of $M\otimes_{\mathbb Z}\mathbb R$ and regard the elements of $M$ as column vectors in $\mathbb R^r$.
The lattice $M$ is called \emph{eutactic} if there exist real numbers $\lambda_x>0$ for $x\in\operatorname{Min}(M)$ such that
\[
I_r =
\sum_{x\in\operatorname{Min}(M)}
\lambda_x xx^{\mathsf T}.
\]
Here $I_r$ denotes the $r\times r$ identity matrix.
If all the $\lambda_x$ can be chosen to be equal, then $M$ is called \emph{strongly eutactic}.
Finally, $M$ is called \emph{extreme} if its similarity class is a local maximum of the Hermite invariant
\[
\gamma(M) := \frac{\min(M)}{\det(M)^{1/r}}
\]
on the space of positive definite quadratic forms modulo positive scalar multiplication. Equivalently, the density of the corresponding lattice sphere packing is locally maximal.
A lattice $M$ is extreme if and only if it is perfect and eutactic (see, for example, \cite[Theorem~3.4.6]{Martinet2003}).

We next recall the relation with spherical designs.
For $m>0$, a nonempty finite subset $X$ of the sphere $S^{r-1}(m):=\{x\in\mathbb R^r\mid(x,x)=m\}$ of radius $\sqrt m$ is called a \emph{spherical $t$-design} if
for every polynomial $p$ of degree at most $t$,
\[
\frac{1}{|X|}
\sum_{x\in X}p(x)
=
\frac{1}{\operatorname{vol}(S^{r-1}(m))}
\int_{S^{r-1}(m)}p(x)\,dx,
\]
where $dx$ denotes the surface measure.  In particular, $M$ is strongly eutactic if and only if
$\operatorname{Min}(M)$ is a spherical $2$-design
(see~\cite[Remark~2.8]{NebeVenkov2013}).

\begin{proposition}\label{prop:extreme88}
The rank \(88\) extremal lattices \((L,B)\) and \((L_0,B_0)\)
are perfect and strongly eutactic.
Consequently, both lattices are extreme.
\end{proposition}

\begin{proof}
By Venkov's theorem~\cite{Venkov1984} (see also \cite[Theorem~4.1]{NebeVenkov2013}),
the set of minimal vectors of each of \(L\) and \(L_0\) is a spherical \(3\)-design.
In particular, it is a spherical \(2\)-design, and hence both lattices are strongly eutactic.
It remains to prove perfection.
We verify this computationally, using the cyclic subgroup
$\langle \zeta_{115}\rangle
   \subseteq \operatorname{Aut}_{\mathbb Z}(L,B)$
and, respectively, the corresponding subgroup of
\(\operatorname{Aut}_{\mathbb Z}(L_0,B_0)\).
For each of the two lattices, we choose \(44\) minimal vectors of norm \(8\)
and take their \(\langle\zeta_{115}\rangle\)-orbits.
Writing the resulting minimal vectors \(x\) in a fixed
\(\mathbb Z\)-basis of the lattice, we form the symmetric rank-one matrices $xx^{\mathsf T}\in \operatorname{Sym}_{88}(\mathbb Z)$. 
Separate exact certificates for \(L\) and \(L_0\) list these \(44\) vectors, the integral order-\(115\) actions, an invertible eigenbasis modulo \(461\), and nonzero maximal minors in all \(115\) symmetric-square character spaces.
Since
\[
\dim_{\mathbb F_{461}}\operatorname{Sym}_{88}(\mathbb F_{461})
=
\frac{88\cdot 89}{2}
=
3916,
\]
and the computation shows that the reductions of the matrices \(xx^{\mathsf T}\) span a \(3916\)-dimensional subspace, they span all of 
$\operatorname{Sym}_{88}(\mathbb F_{461})$. 
Therefore the matrices \(xx^{\mathsf T}\) already span
\(\operatorname{Sym}_{88}(\mathbb Q)\), and hence
\(\operatorname{Sym}_{88}(\mathbb R)\).
Thus the minimal vectors determine the quadratic form uniquely up to scale,
so both \(L\) and \(L_0\) are perfect.
\end{proof}

\begin{proposition}\label{prop:not-4-design}
Let $M$ be an extremal positive definite even unimodular lattice of rank $88$. Then $\operatorname{Min}(M)$ is not a spherical $4$-design.
More precisely, for every $a\in \operatorname{Min}(M)$,
\[
\sum_{x\in \operatorname{Min}(M)}(a,x)^4 \equiv0\pmod 9.
\]
On the other hand, if \(\operatorname{Min}(M)\) were a spherical \(4\)-design, then 
$\sum_{x\in\operatorname{Min}(M)}(a,x)^4
=
261427200
\equiv 6 \pmod 9$. 
\end{proposition}

\begin{proof}
We use the convention
\[
\Theta_M(\tau) :=
\sum_{x\in M}q^{(x,x)/2},
\qquad
q := e^{2\pi \sqrt{-1}\tau}.
\]
Since $M$ is an even unimodular lattice of rank $88$, its theta series is a modular form of weight $44$ for $\operatorname{SL}_2(\mathbb Z)$. Since $M$ is extremal, its minimum is $8$, so the coefficients of $q$, $q^2$, and $q^3$ vanish. The space $M_{44}(\operatorname{SL}_2(\mathbb Z))$ has basis
\[
E_4^{11},\qquad
\Delta E_4^8,\qquad
\Delta^2E_4^5,\qquad
\Delta^3E_4^2.
\]
Imposing the vanishing of the coefficients of $q$, $q^2$, and $q^3$ gives
\[
\Theta_M
=
E_4^{11}
-2640\Delta E_4^8
+1813680\Delta^2E_4^5
-244992000\Delta^3E_4^2.
\]
The cardinality of $\operatorname{Min}(M)$ is the coefficient of $q^4$, namely
\[
|\operatorname{Min}(M)|=168498000.
\]
Fix $a\in \operatorname{Min}(M)$. Thus $(a,a)=8$. For a homogeneous harmonic polynomial $P$ of  degree $d>0$, the weighted theta series
\[
\Theta_{M,P}(\tau)
:=
\sum_{x\in M}P(x)q^{(x,x)/2}
\]
is a modular form of weight $44+d$. Since $d>0$ and $\min(M)=8$, it vanishes to order at least $4$ at the cusp. Hence $\Theta_{M,P}(\tau)$ is divisible by $\Delta^4$.
In particular, if $d < 4$, then $\Theta_{M,P}(\tau) = 0$.
Since the polynomial
\[
P_{2,a}(x)
:=
(a,x)^2-\frac1{11}(x,x)
\]
is harmonic of degree $2$, we have $\Theta_{M,P_{2,a}}=0$. 
Put $S_k(a) := \sum_{x\in \operatorname{Min}(M)}(a,x)^k$. 
Taking the coefficient of $q^4$ gives
\[
S_2(a)-\frac8{11}|\operatorname{Min}(M)|=0,
\]
and therefore $S_2(a)=122544000$.
We next use harmonic polynomials of degree $6$.  The polynomial
\[
P_{6,a}(x)
:=
(a,x)^6-\frac54(x,x)(a,x)^4+\frac{15}{47}(x,x)^2(a,x)^2
-\frac{10}{1081}(x,x)^3
\]
is harmonic. Indeed, using $(a,a)=8$ and
\[
\Delta\bigl((x,x)^{j}(a,x)^k\bigr)
=
2j(88+2j+2k-2)(x,x)^{j-1}(a,x)^k
+
8k(k-1)(x,x)^{j}(a,x)^{k-2},
\]
one checks directly that $\Delta P_{6,a}=0$.
Then $\Theta_{M,P_{6,a}}/\Delta^4$ is a holomorphic modular form of weight $2$. Since
$M_2(\operatorname{SL}_2(\mathbb Z))=0$, we obtain  $\Theta_{M,P_{6,a}}=0$.
Hence the coefficient of $q^4$ gives
\[
S_6(a)
-10S_4(a)
+\frac{960}{47}S_2(a)
-\frac{5120}{1081}|\operatorname{Min}(M)|
=0.
\]
Using the values of $S_2(a)$ and $|\operatorname{Min}(M)|$, this becomes
\begin{equation}
S_6(a)
=
10\bigl(S_4(a)-170496000\bigr).
\label{eq:S6S4}
\end{equation}
Since $M$ is integral, $(a,x)\in\mathbb Z$ for every $x\in \operatorname{Min}(M)$. If $x\ne\pm a$, then
\[
(x-a,x-a)\ge8
\qquad\text{and}\qquad
(x+a,x+a)\ge8.
\]
Since $(x,x)=(a,a)=8$, it follows that $|(a,x)|\le 4$.
Moreover, $(a,x)=8$ occurs only for $x=a$, and $(a,x)=-8$ occurs only for $x=-a$. 
For $0\le j\le4$, put
\[
a_j:=\#\{x\in \operatorname{Min}(M)\mid(a,x)=j\}.
\]
Since $\operatorname{Min}(M)=-\operatorname{Min}(M)$, the numbers of vectors with inner products $j$ and $-j$ are equal. We therefore have
\begin{align*}
a_0+2(a_1+a_2+a_3+a_4)+2 &= 168498000,
\\
2(a_1+4a_2+9a_3+16a_4)+128 &=122544000,
\\
2(a_1+16a_2+81a_3+256a_4)+8192 &=S_4(a),
\\
2(a_1+64a_2+729a_3+4096a_4)+524288 &=S_6(a).
\end{align*}
Substituting \eqref{eq:S6S4} into the last equation and eliminating
$a_1,a_2,a_3$ from these four relations gives
\[
9a_0
=
630a_4+S_4(a)+441661950.
\]
It follows that  $S_4(a)\equiv0\pmod9$.

Suppose now, for contradiction, that $\operatorname{Min}(M)$ is a spherical $4$-design. The polynomial
\[
P_{4,a}(x)
:=
(a,x)^4
-\frac{12}{23}(x,x)(a,x)^2
+\frac{8}{345}(x,x)^2
\]
is harmonic. Since $\operatorname{Min}(M)$ is assumed to be a spherical $4$-design,  $\sum_{x\in \operatorname{Min}(M)}P_{4,a}(x)=0$.
As $(x,x)=8$ for $x\in \operatorname{Min}(M)$, this gives
\[
S_4(a)
-\frac{96}{23}S_2(a)
+\frac{512}{345}|\operatorname{Min}(M)|
=0.
\]
Therefore $S_4(a)=261427200 \equiv6\pmod9$, contradicting  $S_4(a)\equiv0\pmod9$.
\end{proof}

\begin{corollary}\label{cor:L0-not-4-design}
Neither $\operatorname{Min}(L)$ nor $\operatorname{Min}(L_0)$ is a spherical $4$-design.
\end{corollary}

\appendix
\section{\texorpdfstring{The isometry groups of \((L,B)\) and \((L_0,B_0)\)}{The isometry groups of L and L0}}
\label{app:fullgroups}

In this appendix, we determine the isometry groups of \((L,B)\) and \((L_0,B_0)\).
Our argument follows the strategy of \cite[Section~6]{Nebe2013Automorphisms} and uses the classification of finite simple groups, Hall's theorem on groups of symplectic type, and the corrected classification of low-dimensional representations by Hiss--Malle~\cite{HissMalle2001,HissMalle2002}.

\subsection{\texorpdfstring{The automorphism groups of \(E\), \(J\), and \(J_0\)}{The automorphism groups of E, J, and J0}}
\label{grp:sec:factors}

Let \((M,h)\) be a Hermitian \(\mathcal O_F\)-lattice. An element of
$\operatorname{Aut}_{\mathbb Z}(M,\operatorname{Tr}_{F/\mathbb Q}h)$
is $\mathcal O_F$-linear precisely when it commutes with multiplication by
$\alpha$, and is conjugate-linear precisely when it intertwines multiplication
by $\alpha$ with multiplication by $\overline{\alpha}=1-\alpha$.
Moreover,
\begin{align*}
    h(u,v)
=
\frac{
12\operatorname{Tr}_{F/\mathbb Q}h(u,v)
-\operatorname{Tr}_{F/\mathbb Q}h(u,\alpha v)
}{23}
+
\frac{
2\operatorname{Tr}_{F/\mathbb Q}h(u,\alpha v)
-\operatorname{Tr}_{F/\mathbb Q}h(u,v)
}{23}\alpha.
\end{align*}
It follows that
\[
\operatorname{Aut}_{\mathcal O_F}(M,h) = 
\left\{
g\in
\operatorname{Aut}_{\mathbb Z}
\bigl(M,\operatorname{Tr}_{F/\mathbb Q}h\bigr)
\;\middle|\;
g\alpha=\alpha g
\right\},
\]
whereas the conjugate-linear isometries are precisely 
$\left\{
g\in
\operatorname{Aut}_{\mathbb Z}
\bigl(M,\operatorname{Tr}_{F/\mathbb Q}h\bigr)
\;\middle|\;
g\alpha=\overline{\alpha}g
\right\}$. 


\begin{proposition}\label{grp:prop:Egroup}
Put $G_4:=\operatorname{Aut}_{\mathcal O_F}(E,h_4)$. Then $|G_4|=240$, and its GAP Small Groups Library identifier is $\operatorname{SmallGroup}(240,89)$. This is the negative double cover $2^-S_5$ of $S_5$. 
In particular,
\[
Z(G_4)=\langle-1\rangle,\qquad
[G_4,G_4]\cong\operatorname{SL}_2(5),\qquad
G_4/Z(G_4)\cong S_5.
\]
We also have
\[
\{g\in\operatorname{Aut}_{\mathbb Z}
(E,\operatorname{Tr}_{F/\mathbb Q}h_4)
\mid
g\alpha g^{-1}\in\{\alpha,\overline{\alpha}\}\}
\cong G_4\rtimes C_2.
\]
\end{proposition}
\begin{proof}
Computing the centralizer of multiplication by $\alpha$ in $\operatorname{Aut}_{\mathbb Z}
(E,\operatorname{Tr}_{F/\mathbb Q}h_4) \cong \operatorname{Aut}_{\mathbb Z}
(E_8)$, we obtain
$|G_4|=240$ and
$G_4\cong\operatorname{SmallGroup}(240,89)$,
from which the assertions about the center, derived subgroup, and quotient by the center follow. 

Since $E=\mathcal O_{F(\zeta_5)}$  and $\overline{\delta_4}=\delta_4$, complex conjugation
$c\colon E\to E$, $u\mapsto\overline{u}$,
is a conjugate-linear involution. 
Every conjugate-linear isometry is the product of $c$ and an
$\mathcal O_F$-linear isometry. Hence the group of all
$\mathcal O_F$-linear and conjugate-linear isometries is
$G_4 \rtimes\langle c\rangle$
and has order $480$.
\end{proof}

Recall $B_{J} = \operatorname{Tr}_{F/\mathbb{Q}}h_{11}$, and put $B_{J}^{(0)} := \operatorname{Tr}_{F/\mathbb{Q}}h_{11}^{(0)}$.

\begin{proposition}\label{grp:prop:Jgroups}
The following statements hold true. 
\begin{enumerate}[label=(\roman*)]
\item
$\operatorname{Aut}_{\mathcal O_F}(J)
=\operatorname{Aut}_{\mathbb Z}(J,B_J)
\cong C_2\times\operatorname{PSL}_2(23)$.

\item 
$\operatorname{Aut}_{\mathcal O_F}(J_0)
=\operatorname{Aut}_{\mathcal O_F}(J)$
as subgroups of $\operatorname{GL}_F(JF) = \operatorname{GL}_F(J_0F)$.

\item
Moreover,
\[
\operatorname{Aut}_{\mathbb Z}(J_0,B_J^{(0)})
=
\{g\in\operatorname{Aut}_{\mathbb Z}(J_0,B_J^{(0)})
\mid g\alpha g^{-1}\in\{\alpha,\overline{\alpha}\}\}
\cong C_2\times\operatorname{PGL}_2(23).
\]
In particular, $J_0$ admits a conjugate-linear involution inducing the
nontrivial outer automorphism of $\operatorname{PSL}_2(23)$.
\end{enumerate}
\end{proposition}
\begin{proof}
We first record the computational results used below. A computer calculation with the two integral Gram matrices gives $1012$ vectors of norm $12$ for $(J,B_J)$ and $506$ for $(J_0,B_J^{(0)})$. The groups $\operatorname{Aut}_{\mathbb Z}(J,B_J)$ and $\operatorname{Aut}_{\mathbb Z}(J_0,B_J^{(0)})$ have orders $12144$ and $24288$, respectively; in both cases the derived subgroup has order $6072$ and the center has order $2$. These groups act transitively on the vectors of norm $12$, with stabilizer orders $12$ and $48$, respectively. The computation also shows that every element of $\operatorname{Aut}_{\mathbb Z}(J,B_J)$ commutes with multiplication by $\alpha$, while $\operatorname{Aut}_{\mathcal O_F}(J_0)$ has index two in $\operatorname{Aut}_{\mathbb Z}(J_0,B_J^{(0)})$. 
In both cases, the conjugation action of the derived subgroup on its $24$ Sylow $23$-subgroups is faithful and identifies it with the usual fractional-linear action of $\operatorname{PSL}_2(23)$ on $\mathbb P^1(\mathbb F_{23})$. Hence the derived subgroup is isomorphic to $\operatorname{PSL}_2(23)$.

We now prove (i). Since every element of $\operatorname{Aut}_{\mathbb Z}(J,B_J)$ commutes with multiplication by $\alpha$, we have $\operatorname{Aut}_{\mathcal O_F}(J)=\operatorname{Aut}_{\mathbb Z}(J,B_J)$. Its center has order $2$, its derived subgroup is isomorphic to $\operatorname{PSL}_2(23)$, and $12144=2\cdot6072$. Since $\operatorname{PSL}_2(23)$ has trivial center, it follows that
$\operatorname{Aut}_{\mathcal O_F}(J)\cong C_2\times\operatorname{PSL}_2(23)$.

For (ii), an $\mathcal O_F$-linear map stabilizes $J$ if and only if it stabilizes $\overline{\mathfrak p}J=J_0$, since $\overline{\mathfrak p}$ is invertible. 
Therefore $\operatorname{Aut}_{\mathcal O_F}(J_0)=\operatorname{Aut}_{\mathcal O_F}(J)$
as subgroups of $\operatorname{GL}_F(JF)=\operatorname{GL}_F(J_0F)$.

Finally, we prove (iii). Since $J_0$ is principal, say $J_0=\beta\mathcal O_{\mathbb Q(\zeta_{23})}$, the map $x\mapsto(\beta/\bar\beta)\bar x$ gives a conjugate-linear involution of $J_0$. It induces the nontrivial outer automorphism of $\operatorname{PSL}_2(23)$, so the subgroup generated by $\operatorname{PSL}_2(23)$ and this involution is isomorphic to $\operatorname{PGL}_2(23)$. Together with the multiplication by \(-1\), this gives
$\operatorname{Aut}_{\mathbb Z}(J_0,B_J^{(0)})\cong C_2\times\operatorname{PGL}_2(23)$,
since both groups have order \(24288\).
Since $\operatorname{Aut}_{\mathcal O_F}(J_0)$ has index two in $\operatorname{Aut}_{\mathbb Z}(J_0,B_J^{(0)})$, every element of the latter is either $\mathcal O_F$-linear or conjugate-linear. Hence
$\operatorname{Aut}_{\mathbb Z}(J_0,B_J^{(0)})
=
\{g\in\operatorname{Aut}_{\mathbb Z}(J_0,B_J^{(0)})\mid
g\alpha g^{-1}\in\{\alpha,\overline{\alpha}\}\}$.
\end{proof}

\subsection{Rational irreducibility and primitivity}
\label{grp:sec:primitive} 

Recall $G_4 = \operatorname{Aut}_{\mathcal O_F}(E,h_4)$. 
The subgroup of the automorphism group of \(L\) induced by the two tensor
factors is
\[
\frac{
G_4
\times
\operatorname{Aut}_{\mathcal O_F}(J)
}{
\langle(-1,-1)\rangle
}
\cong
G_4
\times
\operatorname{PSL}_2(23),
\]
and the same statement holds with \(J_0\) and \(L_0\).
This group has order \(1457280\). 

Let \(G\) be a finite group and let \(\rho\colon G\to\operatorname{GL}(V)\) be a representation on a finite-dimensional \(\mathbb Q\)-vector space \(V\). The representation \(\rho\) is called primitive if there is no decomposition $V=V_1\oplus\cdots\oplus V_r$  with \(r>1\) and \(V_i\neq0\) such that, for every \(g\in G\) and every \(i\), one has \(\rho(g)V_i=V_j\) for some \(j\).

\begin{lemma}\label{grp:lem:characters}
The representation of
$[G_4,G_4]\times\operatorname{PSL}_2(23)$ on
$L\otimes_{\mathbb Z}\mathbb Q$ is irreducible, and the representation of
$G_4\times\operatorname{PSL}_2(23)$ on
$L\otimes_{\mathbb Z}\mathbb Q$ is primitive. Moreover,
$(L\otimes_{\mathbb Z}\mathbb Q)^{\operatorname{PSL}_2(23)}=0$
and
$\operatorname{End}_{\mathbb Q[\operatorname{PSL}_2(23)]}
(L\otimes_{\mathbb Z}\mathbb Q) =  \operatorname{End}_F(E \otimes_{\mathcal{O}_F} F)$.
The same statements hold with $L_0$ in place of $L$, since
$L\otimes_{\mathbb Z}\mathbb Q\cong L_0\otimes_{\mathbb Z}\mathbb Q$
as $\mathbb Q[G_4\times\operatorname{PSL}_2(23)]$-modules.
\end{lemma}

\begin{proof}
The group $\operatorname{PSL}_2(23)$ has two irreducible complex characters of degree $11$, say $\psi$ and $\bar\psi$. They are Galois conjugate and have character field $F$. Since $11$ is the least degree of a nontrivial complex representation of $\operatorname{PSL}_2(23)$, the representation $J\otimes_{\mathcal O_F}\mathbb C$ is irreducible and affords either $\psi$ or $\bar\psi$, according to the embedding $F\hookrightarrow\mathbb C$. Hence
$J\otimes_{\mathbb Z}\mathbb C\cong\psi\oplus\bar\psi$,
so $J\otimes_{\mathbb Z}\mathbb Q$ is an irreducible $\mathbb Q[\operatorname{PSL}_2(23)]$-module. 
Scalar multiplication by $F$ commutes with the action of $\operatorname{PSL}_2(23)$, and hence
$F\subset \operatorname{End}_{\mathbb Q[\operatorname{PSL}_2(23)]}(J\otimes_{\mathbb Z}\mathbb Q)$.
Moreover,
$\operatorname{End}_{\mathbb Q[\operatorname{PSL}_2(23)]}(J\otimes_{\mathbb Z}\mathbb Q)\otimes_{\mathbb Q}\mathbb C
\cong\mathbb C\oplus\mathbb C$.
Therefore $\operatorname{End}_{\mathbb Q[\operatorname{PSL}_2(23)]}(J\otimes_{\mathbb Z}\mathbb Q)=F$.

It follows from Proposition~\ref{grp:prop:Egroup} that $[G_4,G_4]\cong\operatorname{SL}_2(5)$. The restriction to $[G_4,G_4]$ of the degree-$4$ character afforded by $E$ is the sum of the two faithful irreducible characters of degree $2$ of $\operatorname{SL}_2(5)$, which are conjugate over $\mathbb Q(\sqrt5)$. 
Indeed, on the nine conjugacy classes of $\operatorname{SL}_2(5)$, ordered as $1,2,3,4,5a,5b,6,10a,10b$, the class sizes and the values of the restricted character are
\[
(1,1,20,30,12,12,20,12,12)
\quad\text{and}\quad
(4,-4,-2,0,-1,-1,2,1,1),
\]
respectively. The inner product of this character with itself is $2$, and its inner product with the trivial character is $0$. Since $\operatorname{SL}_2(5)$ has no nontrivial one-dimensional characters, the restriction is the sum of two distinct irreducible characters of degree $2$. These are precisely the two faithful degree-$2$ characters of $\operatorname{SL}_2(5)$.

It follows that the representation $L\otimes_{\mathbb Z}\mathbb C$ of $[G_4,G_4]\times\operatorname{PSL}_2(23)$ decomposes as the direct sum of the four tensor products of the two faithful degree-$2$ characters of $\operatorname{SL}_2(5)$ with the two degree-$11$ characters of $\operatorname{PSL}_2(23)$. These four constituents are distinct, have degree $22$, occur with multiplicity one, and form a single Galois orbit with character field $F(\sqrt5)$, which has degree $4$ over $\mathbb Q$. Hence $L\otimes_{\mathbb Z}\mathbb Q$ is an irreducible $[G_4,G_4]\times\operatorname{PSL}_2(23)$-representation.

Furthermore,
$L\otimes_{\mathbb Z}\mathbb Q\cong(J\otimes_{\mathbb Z}\mathbb Q)^4$
as a $\mathbb Q[\operatorname{PSL}_2(23)]$-module. Hence
$(L\otimes_{\mathbb Z}\mathbb Q)^{\operatorname{PSL}_2(23)}=0$
and
$\operatorname{End}_{\mathbb Q[\operatorname{PSL}_2(23)]}
(L\otimes_{\mathbb Z}\mathbb Q) = \operatorname{End}_F(E \otimes_{\mathcal{O}_F} F)$. 

Suppose that $G_4 \times\operatorname{PSL}_2(23)$ preserves a decomposition
$L\otimes_{\mathbb Z}\mathbb Q=V_1\oplus\cdots\oplus V_b$ with $b>1$ and permutes the subspaces $V_i$. By irreducibility, the action on $\{V_1,\ldots,V_b\}$ is transitive, and hence $b\mid88$. 
Since $\operatorname{PSL}_2(23)$ is normal in $G_4 \times\operatorname{PSL}_2(23)$, its orbits on $\{V_1,\ldots,V_b\}$ all have the same size. By Dickson's subgroup classification (see also~\cite{ATLASPSL223}), the only index of a proper subgroup of $\operatorname{PSL}_2(23)$ not exceeding $88$ is $24$. Hence every nontrivial orbit would have size $24$, which is impossible since $24\nmid b$ and $b\mid88$. Therefore $\operatorname{PSL}_2(23)$ fixes each $V_i$. 
Each $V_i$ is thus a $\mathbb Q[\operatorname{PSL}_2(23)]$-module, so its dimension is divisible by $22$. Since the $V_i$ have equal dimension and their direct sum has dimension $88$, it follows that $b\mid4$.  
The perfect group $[G_4, G_4] \cong\operatorname{SL}_2(5)$ acts trivially on a set of at most four blocks, since every subgroup of
\(S_4\) is solvable.
Each block would therefore be invariant under $[G_4, G_4] \times\operatorname{PSL}_2(23)$, contradicting the  irreducibility proved above.
\end{proof}

\subsection{\texorpdfstring{Centralizers and normality of \(\operatorname{PSL}_2(23)\)}{Centralizers and normality of PSL2(23)}}
\label{grp:sec:normality} 

\begin{lemma}\label{grp:lem:centralizer}
We have 
\[
C_{\operatorname{Aut}_{\mathbb Z}(L,B)}
\bigl(\operatorname{PSL}_2(23)\bigr)
= G_4, \qquad 
C_{\operatorname{Aut}_{\mathbb Z}(L_0,B_0)}
\bigl(\operatorname{PSL}_2(23)\bigr)
=
G_4.
\]
\end{lemma}

\begin{proof}
Let $M\in\{J,J_0\}$. 
By Lemma~\ref{grp:lem:characters}, $\operatorname{End}_{\Q[\operatorname{PSL}_2(23)]}(M\otimes_{\Z}\Q)=F$.
Thus $\operatorname{End}_{\Z[\operatorname{PSL}_2(23)]}(M)$ is an order in $F$ containing the maximal order $\mathcal O_F$, and hence equals $\mathcal O_F$.
Since $E$ is free over $\mathcal O_F$, it follows that
\[
\operatorname{End}_{\Z[\operatorname{PSL}_2(23)]}(E\otimes_{\mathcal O_F}M)
=\operatorname{End}_{\mathcal O_F}(E),
\]
under the natural identification $T\mapsto T\otimes1$.
Consequently, every element of the centralizer acts as $T\otimes1$ with $T\in\operatorname{GL}_{\mathcal O_F}(E)$.
Both centralizers are therefore $G_4$.
\end{proof}

Recall that the Fitting subgroup \(F(G)\) of a finite group \(G\) is its largest nilpotent normal subgroup. 

\begin{lemma}\label{grp:lem:fitting}
The Fitting subgroups of $\operatorname{Aut}_{\mathbb Z}(L,B)$ and $\operatorname{Aut}_{\mathbb Z}(L_0,B_0)$ are both equal to $\langle -1\rangle$. 
\end{lemma}

\begin{proof}
We prove the assertion for
\(\operatorname{Aut}_{\mathbb Z}(L,B)\); the same proof applies to
\(\operatorname{Aut}_{\mathbb Z}(L_0,B_0)\). 
Let \(H\) be a normal subgroup of \(\operatorname{Aut}_{\mathbb Z}(L,B)\). Since \(H\) is normal, \(\operatorname{Aut}_{\mathbb Z}(L,B)\) permutes the  isotypic components of \(L\otimes_{\mathbb Z}\mathbb Q\) for \(H\). 
It follows from Lemma \ref{grp:lem:characters} that there is only one such component. 
In particular, if \(H\) is abelian, then \(L\otimes_{\mathbb Z}\mathbb Q\) is a direct sum of copies of a faithful irreducible  \(\mathbb{Q}[H]\)-module \(V\). Since \(H\) is abelian, the representation \(H\hookrightarrow\operatorname{GL}(V)\) identifies \(H\) with a finite subgroup of the multiplicative group of a number field. Hence \(H\) is cyclic. Therefore every abelian normal subgroup of \(\operatorname{Aut}_{\mathbb Z}(L,B)\) is cyclic. 

For a prime \(p\), let \(O_p(G)\) denote the largest normal \(p\)-subgroup of a finite group \(G\). Every abelian characteristic subgroup of \(O_p(\operatorname{Aut}_{\mathbb Z}(L,B))\) is therefore cyclic. 
Suppose that \(p\) is odd. By Hall's theorem~\cite[p.~357]{Huppert1967}, if \(O_p(\operatorname{Aut}_{\mathbb Z}(L,B))\) is noncyclic, then it is a central product of an extraspecial \(p\)-group and a cyclic \(p\)-group. 
After extension to \(\mathbb C\), any faithful irreducible rational representation of such a group contains an irreducible complex representation of degree divisible by \(p\), with character field of degree divisible by \(p-1\). 
Since its rational degree divides \(88\), the only possible prime is \(p=11\). But in this case its rational degree is at least \(11\cdot10=110\), a contradiction. 
Hence \(O_p(\operatorname{Aut}_{\mathbb Z}(L,B))\) is cyclic for every odd prime \(p\). Its automorphism group is therefore abelian, so the perfect group \(\operatorname{PSL}_2(23)\) centralizes it.

For \(p=2\), let
\(P=O_2(\operatorname{Aut}_{\mathbb Z}(L,B))\).
Since \(P\) is normal, the argument above shows that the restriction of
\(L\otimes_{\mathbb Z}\mathbb Q\) to \(P\) is a direct sum of copies of a single faithful irreducible rational \(P\)-representation \(V\).
Every irreducible complex \(P\)-representation occurring in
\(V\otimes_{\mathbb Q}\mathbb C\)
is faithful and has degree \(2^a\), since \(P\) is a \(2\)-group.
Since \(2^a\) divides \(88\), we have \(a\le3\). 
Hall's classification shows that \(P\) can be generated by at most \(2a+2\le8\) elements. By the Burnside basis theorem, the kernel of the natural homomorphism
\(\operatorname{Aut}(P)\to\operatorname{Aut}(P/\Phi(P))\)
is a \(2\)-group, while its image embeds into \(\operatorname{GL}_8(2)\). Since the multiplicative order of \(2\) modulo \(23\) is \(11\), none of \(2^i-1\) for \(1\le i\le8\) is divisible by \(23\). Hence \(23\nmid|\operatorname{Aut}(P)|\). Since \(\operatorname{PSL}_2(23)\) is simple, its conjugation action on \(P\) is therefore trivial.

Consequently every normal \(p\)-subgroup of
\(\operatorname{Aut}_{\mathbb Z}(L,B)\) is contained in $C_{\operatorname{Aut}_{\mathbb Z}(L,B)}
\bigl(\operatorname{PSL}_2(23)\bigr)
=
G_4$. 
The Fitting subgroup is therefore a nilpotent normal subgroup of
\(G_4\).
Since $G_4/Z(G_4) \cong S_5$ and \(S_5\) has trivial Fitting subgroup, 
$F\bigl(\operatorname{Aut}_{\mathbb Z}(L,B)\bigr) \subset Z(G_4)=\langle-1\rangle$. 
The reverse inclusion follows because \(-1\) is central.
\end{proof}

\begin{lemma}\label{grp:lem:list}
Suppose that \(Q\) is quasisimple, \(23\mid |Q/Z(Q)|\), and \(Q\) admits a faithful irreducible complex representation \(\rho\) whose degree divides \(88\). Then the possibilities are
\[
\begin{array}{c|c}
Q & \deg(\rho) \\ \hline
\operatorname{PSL}_2(23) & 11,22\\
\operatorname{SL}_2(23) & 22\\
M_{23} & 22
\end{array}
\qquad
\begin{array}{c|c}
Q & \deg(\rho) \\ \hline
A_{23} & 22\\
A_{45} & 44\\
A_{89} & 88
\end{array}
\]
Here \(M_{23}\) denotes the Mathieu group. 
\end{lemma}

\begin{proof}
Apply Hiss--Malle's theorem~\cite{HissMalle2001}, together with its corrigendum~\cite{HissMalle2002}.  Since the degree of \(\rho\) divides \(88\), only the degrees \(2,4,8,11,22,44,88\) need to be considered. Among the nongeneric cases, the only group whose simple quotient has order divisible by \(23\) is \(M_{23}\), occurring in degree \(22\). 

For the \(L_2(q)\) families, the relevant degrees are \(q\), \(q\pm1\), and \((q\pm1)/2\), subject to the usual conditions on the group and its covering groups. Requiring such a degree to divide \(88\) and \(23\mid q(q^2-1)\) gives \(q=23\). The resulting faithful representations of \(\operatorname{PSL}_2(23)\) and \(\operatorname{SL}_2(23)\) have the degrees listed above.

Finally, suppose that \(Q/Z(Q)\cong A_n\). Since \(23\mid |A_n|\), we have \(n\ge23\). For \(n\ge23\), the smallest nontrivial irreducible complex degree is \(n-1\), while the next possible degree is at least \(n(n-3)/2>88\). The nontrivial double covers contribute no additional cases, since for \(n\ge23\) there is no faithful spin representation of degree at most \(88\). Thus \(n-1\mid88\), which gives \(n=23,45,89\). 
\end{proof}

\begin{proposition}\label{grp:prop:Nnormal}
The subgroup $\operatorname{PSL}_2(23)$ is normal in both
$\operatorname{Aut}_{\mathbb Z}(L,B)$ and
$\operatorname{Aut}_{\mathbb Z}(L_0,B_0)$.
    
\end{proposition}

\begin{proof}
We give the proof for
\(\operatorname{Aut}_{\mathbb Z}(L,B)\); the proof for
\(\operatorname{Aut}_{\mathbb Z}(L_0,B_0)\) is identical.

Let $Q_1,\ldots,Q_r$ be its quasisimple subnormal subgroups.
By Lemma~\ref{grp:lem:fitting}, the generalized Fitting subgroup of $\operatorname{Aut}_{\mathbb Z}(L,B)$ is
\[
F(\operatorname{Aut}_{\mathbb Z}(L,B))Q_1\cdots Q_r=\pm Q_1\cdots Q_r.
\] 
Since \(Q_1\cdots Q_r\) is normal, Lemma~\ref{grp:lem:characters} implies that the restriction of
\(L\otimes_{\mathbb Z}\mathbb Q\) to \(Q_1\cdots Q_r\) is a direct sum of copies of a single faithful irreducible  representation. 
The irreducible complex representations occurring after extension to \(\mathbb C\) are Galois conjugate and faithful. Their common degree divides \(88\) and is the product of the degrees of faithful irreducible representations of \(Q_1,\ldots,Q_r\). Since each of these degrees is at least \(2\) and \(88=2^3\cdot11\), we have \(r\le4\). 
The group $\operatorname{PSL}_2(23)$ acts by conjugation on the set $\{Q_1,\ldots,Q_r\}$. 
Since the smallest index of a proper subgroup of \(\operatorname{PSL}_2(23)\) is \(24\), it cannot act nontrivially on the set \(\{Q_1,\ldots,Q_r\}\) with \(r\le4\). 
By the Schreier conjecture, which follows from the classification of finite simple groups, the outer automorphism group of \(Q_i/Z(Q_i)\) is solvable. Hence the conjugation action of $\operatorname{PSL}_2(23)$ on each $Q_i$ is inner.

Since the distinct \(Q_i\) commute with one another, for every
\(g\in\operatorname{PSL}_2(23)\) we may choose \(q_i\in Q_i\) such that conjugation by \(g\) on \(Q_i\) agrees with conjugation by \(q_i\), and then \(q=q_1\cdots q_r\) has the property that \(gq^{-1}\) centralizes \(\pm Q_1\cdots Q_r\). 
It follows from the self-centralizing property of the generalized Fitting subgroup that $gq^{-1} \in \pm Q_1\cdots Q_r$. 
We therefore obtain
$\operatorname{PSL}_2(23) \subset \pm Q_1\cdots Q_r$.
Since $\operatorname{PSL}_2(23)$ is perfect, it follows that
$\operatorname{PSL}_2(23)\subset Q_1\cdots Q_r$.

Since
\[
(Q_1\cdots Q_r)/Z(Q_1\cdots Q_r)
\cong
\prod_{i=1}^r Q_i/Z(Q_i),
\]
for each \(i\) we obtain a homomorphism $\operatorname{PSL}_2(23)\to Q_i/Z(Q_i)$.
Since \(\operatorname{PSL}_2(23)\) is simple, each such homomorphism is either trivial or injective.
 At least one is nontrivial, and its target has order divisible by $23$. By Lemma~\ref{grp:lem:list}, the corresponding $Q_i$ has a faithful irreducible complex representation of degree at least $11$. There cannot be two such factors, since $11^2>88$. Hence there is a unique $Q_i$ on which $\operatorname{PSL}_2(23)$ has a nontrivial projection. Since all other projections are trivial and $\operatorname{PSL}_2(23)$ is perfect, we obtain $\operatorname{PSL}_2(23)\subset  Q_i$.

It remains to determine $Q_i$ using Lemma~\ref{grp:lem:list}. If $Q_i\cong M_{23}$ or $A_{23}$, the inclusion $\operatorname{PSL}_2(23)\subset Q_i \hookrightarrow S_{23}$ would give a faithful permutation representation of $\operatorname{PSL}_2(23)$ of degree $23$, which is impossible. Also, $\operatorname{SL}_2(23)$ cannot contain $\operatorname{PSL}_2(23)$, since such a subgroup would have index two, contradicting the perfectness of $\operatorname{SL}_2(23)$. 
Suppose that \(Q_i\cong A_{45}\) or \(A_{89}\). Consider the natural action of \(A_n\) on \(\{1,\ldots,n\}\), where \(n=45\) or \(89\). By the subgroup classification of \(\operatorname{PSL}_2(23)\), every nontrivial orbit of \(\operatorname{PSL}_2(23)\) of size at most \(89\) has size \(24\). Since neither \(45\) nor \(89\) is divisible by \(24\), the restricted action on \(\{1,\ldots,n\}\) is not transitive. 
If \(\operatorname{PSL}_2(23)\) has \(m>1\) orbits on \(\{1,\ldots,n\}\), then the fixed subspace of the permutation representation on \(\mathbb C^n\) has dimension \(m\). After removing the one-dimensional space of constant vectors, the resulting degree \(n-1\) representation has a fixed subspace of dimension \(m-1>0\). Hence 
$(L\otimes_{\mathbb Z}\mathbb Q)^{\operatorname{PSL}_2(23)}\neq0$, 
contradicting Lemma~\ref{grp:lem:characters}.

Thus $Q_i\cong\operatorname{PSL}_2(23)$, and hence $Q_i=\operatorname{PSL}_2(23)$. Moreover, it is the unique quasisimple subnormal subgroup whose simple quotient has order divisible by $23$. 
Therefore it is preserved under conjugation by every element of $\operatorname{Aut}_{\mathbb Z}(L,B)$, and hence
$\operatorname{PSL}_2(23)\triangleleft\operatorname{Aut}_{\mathbb Z}(L,B)$. 
\end{proof}

\begin{lemma}\label{grp:lem:semilinear}
Every element of $\operatorname{Aut}_{\mathbb Z}(L,B)$ and
$\operatorname{Aut}_{\mathbb Z}(L_0,B_0)$ is either $\mathcal O_F$-linear or conjugate-linear.
\end{lemma}

\begin{proof}
We prove the assertion for \(L\); the proof for \(L_0\) is identical.
By Proposition~\ref{grp:prop:Nnormal}, every element of
\(\operatorname{Aut}_{\mathbb Z}(L,B)\) normalizes
\(\operatorname{PSL}_2(23)\). Hence conjugation preserves
\[
\operatorname{End}_{\mathbb Q[\operatorname{PSL}_2(23)]}
(L\otimes_{\mathbb Z}\mathbb Q)
\cong
\operatorname{End}_F(E\otimes_{\mathcal O_F}F). 
\]
It therefore preserves the center \(F\) of this algebra. Thus every element of
\(\operatorname{Aut}_{\mathbb Z}(L,B)\) induces an automorphism of \(F\), and is either
\(\mathcal O_F\)-linear or conjugate-linear.
\end{proof}

\subsection{\texorpdfstring{Determination of the isometry groups of \((L,B)\) and \((L_0,B_0)\)}{Determination of the isometry groups of L and L0}}
\label{grp:sec:full}

\begin{lemma}\label{grp:lem:steinitz}
The Steinitz classes of \(J\) and \(L\) are both
\([\overline{\mathfrak p}]\ne1\), whereas \(J_0\) and \(L_0\) have trivial Steinitz class.
In particular, \(L\) admits no conjugate-linear module automorphism.
\end{lemma}

\begin{proof}
Let $\operatorname{St}(M)$ denote the Steinitz class of a fintely generated torsion free $\mathcal{O}_F$-module of $M$. 
The relative discriminant of $\Q(\zeta_{23})/F$ is the principal ideal $(\sqrt{-23})^{10}$.
It follows from $\Cl(F)\cong C_3$ that $\operatorname{St}(\mathcal O_{\Q(\zeta_{23})})=1$.
The principality of $J_0$ over $\mathcal O_{\Q(\zeta_{23})}$ therefore gives $\operatorname{St}(J_0)=1$.
Now $J_0=\overline{\mathfrak p}J$ and $\operatorname{rank}_{\mathcal O_F}J=11$, so
\[
1=\operatorname{St}(J_0)=[\overline{\mathfrak p}]^{11}\operatorname{St}(J),
\qquad \operatorname{St}(J)=[\overline{\mathfrak p}]\ne1.
\]
As $E$ is free of rank $4$, we have $\operatorname{St}(L)=\operatorname{St}(J)^4=[\overline{\mathfrak p}]$. 
\end{proof}

\begin{theorem}\label{grp:thm:fullgroups}
We have
\[
\operatorname{Aut}_{\mathbb Z}(L,B) = \operatorname{Aut}_{\mathcal O_F}(L)
= 
\operatorname{Aut}_{\mathcal O_F}(E,h_4)
\times
\operatorname{PSL}_2(23). 
\]
Moreover, $\operatorname{Aut}_{\mathcal O_F}(L_0) = \operatorname{Aut}_{\mathcal O_F}(E,h_4)$ and 
\[
\operatorname{Aut}_{\mathbb Z}(L_0,B_0)
\cong
\left(
\operatorname{Aut}_{\mathcal O_F}(E,h_4)
\times
\operatorname{PSL}_2(23)
\right)
\rtimes C_2.
\]
Here the nontrivial element of $C_2$ acts on
$\operatorname{Aut}_{\mathcal O_F}(E,h_4)$ by complex conjugation and on
$\operatorname{PSL}_2(23)$ by its nontrivial outer automorphism.
\end{theorem}

\begin{proof}
By Proposition~\ref{grp:prop:Nnormal}, $\operatorname{PSL}_2(23)$ is normal in both
$\operatorname{Aut}_{\mathbb Z}(L,B)$ and
$\operatorname{Aut}_{\mathbb Z}(L_0,B_0)$.
Conjugation therefore gives a homomorphism from each of these groups to
$\operatorname{Aut}(\operatorname{PSL}_2(23))\cong\operatorname{PGL}_2(23)$.
By Lemma~\ref{grp:lem:centralizer}, the kernel is
$G_4$, which has order $240$.
Hence
\[
|\operatorname{Aut}_{\mathbb Z}(L,B)|,\,
|\operatorname{Aut}_{\mathbb Z}(L_0,B_0)|
\le
240\cdot|\operatorname{PGL}_2(23)|
=
240\cdot12144.
\]
For \(L\), Lemmas~\ref{grp:lem:semilinear} and~\ref{grp:lem:steinitz} give
\(\operatorname{Aut}_{\mathbb Z}(L,B)=\operatorname{Aut}_{\mathcal O_F}(L)\).
Let \(\psi\) and \(\bar\psi\) be the two Galois-conjugate irreducible characters of degree \(11\) of \(\operatorname{PSL}_2(23)\). For a fixed embedding \(F\hookrightarrow\mathbb C\), the representation
\(L\otimes_{\mathcal O_F}\mathbb C\) affords \(4\psi\), while the conjugate embedding affords \(4\bar\psi\). Since the nontrivial outer automorphism of \(\operatorname{PSL}_2(23)\) exchanges \(\psi\) and \(\bar\psi\), an \(\mathcal O_F\)-linear automorphism of \(L\) cannot induce it. Therefore the image of $\operatorname{Aut}_{\mathcal O_F}(L) \to \operatorname{PGL}_2(23)$ is contained in
\(\operatorname{PSL}_2(23)\), and hence
\[
|\operatorname{Aut}_{\mathbb Z}(L,B)|\le240\cdot6072.
\]
The tensor construction already gives a subgroup
isomorphic to
$G_4 \times\operatorname{PSL}_2(23)$
of this order. Hence
\[
\operatorname{Aut}_{\mathbb Z}(L,B)
=
\operatorname{Aut}_{\mathcal O_F}(L)
\cong
G_4 
\times
\operatorname{PSL}_2(23).
\]
The argument of the proof of Proposition~\ref{grp:prop:Jgroups}(ii) gives 
$\operatorname{Aut}_{\mathcal O_F}(L_0)=\operatorname{Aut}_{\mathcal O_F}(L)$. 
Since $E$ and $J_0$ admit conjugate-linear involutions, their tensor product gives an element of
$\operatorname{Aut}_{\mathbb Z}(L_0,B_0)$ inducing the nontrivial outer automorphism of
$\operatorname{PSL}_2(23)$.
Thus $\operatorname{Aut}_{\mathbb Z}(L_0,B_0)$ contains a subgroup of order
$240\cdot12144$, which is the upper bound above. Hence 
$\operatorname{Aut}_{\mathbb Z}(L_0,B_0)
\cong
\left(
G_4 
\times
\operatorname{PSL}_2(23)
\right)
\rtimes C_2$. 
\end{proof}

\bibliographystyle{amsalpha}
\bibliography{88extremal}

\end{document}